\documentclass[12pt]{amsart}
\usepackage{amscd}      % to be able to use "CD" environment
\usepackage{amssymb}
\usepackage{amsmath, amsthm, graphics}
\usepackage{xypic}      % to be able to use "diagram" environment of
\LaTeXdiagrams          % xypic to typeset commutative diagrams
\usepackage[all]{xy}
\xyoption{2cell} \UseAllTwocells \xyoption{frame} \CompileMatrices
\allowdisplaybreaks[3]

\usepackage{latexsym}
\usepackage{epsfig}
\usepackage{amsfonts}
\usepackage{enumerate}
\usepackage{times}

 \usepackage[mathscr]{euscript}
  \usepackage[normalem]{ulem}
  \usepackage{url}

\newtheorem{prop}{Proposition}[section]
\newtheorem{lem}[prop]{Lemma}

\newtheorem{cor}[prop]{Corollary}
\newtheorem{thm}[prop]{Theorem}

\newtheorem{theorem}[subsection]{Theorem}

\theoremstyle{definition}

\theoremstyle{remark}

\theoremstyle{remark}

\numberwithin{equation}{section}

\newcommand{\Y}{\mathcal{Y}}

\newcommand{\A}{\mathbb{A}}
\newcommand{\B}{\mathcal{B}}

\newcommand{\qcoh}[1]{\mathrm{QCoh}(#1)}
\newcommand{\coh}[1]{\mathrm{Coh}(#1)}
\renewcommand{\B}{\mathscr{B}}
\renewcommand{\A}{\mathscr{A}}
\newcommand{\modd}[1]{#1\mathrm{-mod}}
\newcommand{\Mod}[1]{#1\mathrm{-Mod}}
\newcommand{\rep}[1]{#1\mathrm{-rep}}
\newcommand{\Rep}[1]{#1\mathrm{-Rep}}

\def\<{\left\langle}
\def\>{\right\rangle}
\begin{document}

\title{Stacks associated to abelian tensor categories}
\author{Yu-Han Liu}
\address{Department of Mathematics, Princeton University, Fine Hall, Washington Road, Princeton NJ 08544}
\email{yuliu@math.princeton.edu}

\author{Hsian-Hua Tseng}
\address{Department of Mathematics, Ohio State University, 100 Math Tower, 231 West 18th Avenue, Columbus, OH 43210}
\email{hhtseng@math.ohio-state.edu}

\date{\today}

\begin{abstract}
For an abelian tensor category a stack is constructed.  As an application we show that our construction can be used to recover a quasi-compact separated scheme from the category of its quasi-coherent sheaves. In another application, we show how the ``dual stack'' of the classifying stack $BG$ of a finite group $G$ can be obtained by altering the tensor product on the category $\rep{G}$ of $G$-representations. Using glueing techniques we show that the dual pair of a $G$-gerbe, in the sense of \cite{TT}, can be constructed by glueing local dual stacks.
\end{abstract}

\maketitle
%\tableofcontents

\section{Introduction}

Let $X$ be a Noetherian scheme over an algebraically closed fields $k$ of characteristic $0$. It is known that the category $\qcoh{X}$ of quasi-coherent sheaves on $X$ uniquely determines $X$; see \cite[page~447]{gabriel}, and \cite[Corollary~4.4]{rouquier} for the version with coherent sheaves. Motivated by this, in this paper we study the (re)construction of geometry from abelian categories. Our approach is motivated by the following consideration. Let $\text{Aff}_k$ be the category of affine $k$-schemes. A scheme $X$ is equivalent to its {\em functor of points} $$\text{Aff}_k\to \text{Sets}, \quad S\mapsto \text{Hom}(S,X).$$ Given a morphism $\phi: S\to X$ of schemes, one can consider the pull-back functor $\phi^*: \qcoh{X}\to \qcoh{S}$. $\phi^*$ is a {\em symmetric monoidal functor}, i.e. $\phi^*$ is compatible with tensor products of sheaves. A version of the Tannakian duality theorem (see \cite{lurie}) states that when $X$ is a quasi-compact separated scheme, $\phi\mapsto \phi^*$ defines an equivalence \[\mathrm{Hom}(S,X)\simeq \mathrm{Hom}_\otimes (\qcoh{X}, \qcoh{S})\] where the right-hand side denotes the category of symmetric monoidal functors which carry flat objects to flat objects and preserve colimits.

The discussion above suggests the following general construction. Let $(A,\otimes)$ be an abelian tensor category. Define a functor $$\uline{(A,\otimes)}: \text{Aff}_k\to \text{Groupoids}, \quad \underline{(A,\otimes)}(S):=\mathrm{Hom}_\otimes(A, \qcoh{S}).$$
\begin{theorem}
The functor $\underline{(A,\otimes)}$ is represented by a stack.
\end{theorem}
The above Theorem is a special case of our main Theorem, which is described and proved in Section \ref{sec:stack} as Theorem \ref{thm:stack}.

We illustrate our construction in several examples. Suppose that $(A, \otimes)=(\qcoh{X},\otimes)$ is the category of quasi-coherent sheaves on a quasi-compact separated $k$-scheme $X$ with $\otimes$ being the tensor product of sheaves. In this case we have
\begin{theorem}
There is an isomorphism $$\underline{(\qcoh{X},\otimes)}\simeq X.$$
\end{theorem}
This is proved in Section \ref{sec:scheme} as Theorem \ref{thm:reconst} in case $X$ is affine, and Theorem \ref{thm:reconst2} for general quasi-compact separated $X$. We expect that this Theorem also holds true for $X$ being a geometric stack (in the sense of \cite[Definition 3.1]{lurie}).

The stack $\uline{(A,\otimes)}$ we construct is sensitive to the tensor structure $\otimes$. To illustrate this, we consider the example $A=\rep{G}$, the category of finite dimensional representations of a finite group $G$ over $k$. $\rep{G}$ may be interpreted as the category $\coh{BG}$ of coherent sheaves on the classifying stack $BG$. In this point of view tensor product of sheaves gives $\rep{G}$ a tensor structure denoted by $\otimes_G$, and we have
\begin{theorem}[see Theorem \ref{thm:BG}]
There is an isomorphism $$\underline{(\rep{G},\otimes_G)}\simeq BG.$$
\end{theorem}
We expect that the above Theorem to hold also for not necessarily finite groups $G$, such as linear algebraic groups.

On the other hand, Schur's Lemma implies that $\rep{G}$ is equivalent (as abelian categories) to the direct sum of $\#\hat G$ copies\footnote{$\hat G$ denotes the group of characters on $G$.} of the category $\text{vect}_k$ of finite dimensional $k$-vector spaces. Componentwise tensor product of vector spaces then gives a different tensor structure on $\rep{G}$, which is denoted by $\otimes_Z$. In this case we have
\begin{theorem}[See Section \ref{sec:char}]
$\underline{(\rep{G},\otimes_Z)}$ is isomorphic to a disjoint union of $\#\hat G$ points.
\end{theorem}

Certainly $BG$ is very different from a disjoint union of $\#\hat G$ points. However they are closely related. In fact the correspondence between them is the simplest example of the {\em gerbe duality} studied in \cite{TT}. For a $G$-gerbe $\Y\to \B$ the gerbe duality asserts that various geometric properties of $\Y$ are equivalent to geometric properties of a ``dual space'' $\widehat{\Y}$ {\em twisted by} a $\mathbb{C}^*$-valued $2$-cocycle $c$. In particular, it is shown in \cite{TT} that the category of sheaves on $\Y$ is equivalent to the category of $c$-twisted sheaves on $\widehat{\Y}$. The $BG$ example suggests that it may be possible to realize the $G$-gerbe $\Y$ and its dual $(\widehat{\Y}, c)$ using stacks associated to the category of sheaves on $\Y$ with different tensor structures. In Section \ref{sec:gerbe} we carry out a construction of this nature. Locally on $\Y$, the stacks we construct realize the $G$-gerbe and its dual. We then obtain $\Y$ and $(\widehat{\Y}, c)$ by glueing local duals.

\subsection*{Outlook}
It is very interesting to study the stacks constructed in this paper for other examples of abelian tensor categories. For instance, it will be very desirable to describe the stacks associated to the category of rational Hodge structures, or the category of representations of a quiver.

One may hope that properties of the stack $\underline{(A,\otimes)}$ can reflect properties of the category $(A,\otimes)$. For this reason it is interesting to study geometric properties of $\underline{(A,\otimes)}$. For example, it will be interesting to find criteria for algebraicity of $\underline{(A,\otimes)}$. We plan to pursue this elsewhere.

The rest of this paper is organized as follows.  In section \ref{sec:ABstack} we define the category $\uline{(A,\otimes)}$ of tensor functors from a fixed abelian tensor category $A$ into a varying family $\B$ of abelian tensor categories, and we show that if the target categories $\B$ form a fibred category or (pre-)stack, then so does $\uline{(A,\otimes)}$.  We also explain how objects in $A$ naturally give rise to sheaves on $\uline{(A,\otimes)}$.

In section \ref{sec:scheme} we explain how schemes can be reconstructed as stacks of the form $\uline{(A,\otimes)}$ when $A$ is taken to be the category of quasi-coherent sheaves on the scheme.  In section \ref{sec:group} we consider the example $A=\rep{G}$ where $G$ is a finite group, and we study $\uline{(A,\otimes)}$ equipped with different tensor structures.

In Section \ref{sec:gerbe} we construct stacks locally of the form $\uline{(A,\otimes)}$ using 2--descent.  As an example we show how the gerbe duality of \cite{TT} can be understood via this construction:  Indeed, a $G$-gerbe $\mathcal Y$ and the underlying space of its dual $\hat{\mathcal Y}$ can both locally be realized as stacks of the form $\uline{(A,\otimes)}$ with the same abelian category $A=\rep{G}$ and ``dual'' choices of tensor products.

In Appendix \ref{sec:twisting} we give a brief review of the construction of stacks using 2--descent.

\section{Stack associated to an abelian tensor category}\label{sec:ABstack}

\subsection{Category fibred in abelian tensor categories}

\subsubsection{} Let $\mathfrak S$ be a category, and let $\pi:\B\rightarrow \mathfrak S$ be a category fibred in abelian tensor categories.  This means first of all that $\pi$ is a functor making $\B$ a fibred category; in particular for every morphism $f: S'\rightarrow S$ in $\mathfrak S$ there is a \emph{canonical} functor between the fibre categories \[f^*: \B_S\rightarrow \B_{S'}.\]

Every fibre category $\B_S$ is required to be an abelian tensor category, which means that it is an abelian category along with a symmetric monoidal category structure $\otimes_S$ \cite[Chapter~XI.1]{maclane_categories}, such that the functor $x\mapsto y\otimes_S x$ is additive and preserves (small) colimits for every $y\in\B_S$.  We denote by $u_S$ the unit object in $\B_S$.

The assumption above in particular implies that the functor $x\mapsto y\otimes_S x$ is right exact, and that the zero object $0\in\B_S$ satisfies $0\otimes_Sx\cong 0$ for every $x\in \B_S$.

Lastly, the pull-back functors $f^*$ are required to be symmetric strong monoidal \cite[Chapter~XI.2]{maclane_categories} and preserves colimits; in particular they preserve direct sums.  The condition of being symmetric strong monoidal means that there are isomorphisms \[f^*(y)\otimes_{S'}f^*(x)\stackrel{\cong}{\longleftarrow} f^*(x)\otimes_{S'} f^*(y)\stackrel{\cong}{\longrightarrow} f^*(x\otimes_S y)\] and \[u_{S'}\stackrel{\cong}{\longrightarrow} f^*(u_S)\]  for every $x,y\in\B_S$ which satisfy some compatibility, or coherence conditions in the form of commutative diagrams; to simplify notations we will often suppress these isomorphisms and canonically identify these objects in $\B_{S'}$.

We will call symmetric strong monoidal functors preserving colimits simply \emph{tensor functors}.  In particular they are right exact functors.

\subsubsection{}\label{par:exqcoh}  \emph{Example.}  Fix a commutative ring $\Lambda$ with identity.  Let $\mathfrak S$ be the category of affine $\Lambda$-schemes.  Let $\B$ be the category of pairs $(S,\mathcal F)$ with $S\in\mathfrak S$ and $\mathcal F$ a quasi-coherent sheaf on $S$.  The obvious forgetful functor $\pi: \B\rightarrow \mathfrak S$ is a category fibred in abelian tensor categories with the sheaf tensor product on each $\B_S=\qcoh{S}$.

\subsubsection{}\label{par:exqrep}  \emph{Example.}  Fix a field $k$.  Let $\mathfrak S$ be the category of finite quivers (without relations).  Let $\B$ be the category of pairs $(Q,V)$ with $Q\in\mathfrak S$ and $V$ a finite dimensional $Q$-representation over $k$.  The obvious forgetful functor $\pi: \B\rightarrow \mathfrak S$ is a category fibred in abelian tensor categories with the vertex-wise tensor product on each $\B_Q=\rep{Q}_k$.

\subsubsection{}\label{par:exgrep}  \emph{Example.}  Fix a field $k$.  Let $\mathfrak S$ be the category of finite groups.  Let $\B$ be the category of pairs $(G,V)$ with $G\in\mathfrak S$ and $V$ a finite dimensional $G$-representation over $k$.  The obvious forgetful functor $\pi: \B\rightarrow \mathfrak S$ is a category fibred in abelian tensor categories with the representation tensor product on each $\B_Q=\rep{G}_k$.

\subsection{Category fibred in groupoids associated to an abelian tensor category}

\subsubsection{}\label{}  Let $(A,\otimes)$ be an abelian tensor category with unit object $u$, and let $\pi: \B\rightarrow \mathfrak S$ be a category fibred in abelian tensor categories as above.

Define a category $\uline{(A,\otimes)}(\B)$, sometimes just $\uline{A}(\B)$, $\uline{(A,\otimes)}$, or $\uline A$ whenever the tensor product and/or the category $\B$ in question are clear, whose objects are pairs $(S,F)$ where $S\in\mathfrak S$ and $F:A\rightarrow \B_S$ is a \emph{tensor functor}.  A morphism $(S',F')\rightarrow (S,F)$ is a pair $(f,\phi)$ where $f: S'\rightarrow S$ is a morphism in $\mathfrak S$ and $\phi$ is a \emph{symmetric monoidal} natural \emph{isomorphism} between tensor functors \[\phi: F'\rightarrow f^*\circ F\] from $A$ to $\B_{S'}$:

\centerline{\xymatrix{&&A \ar[ddll]_-{F'} \ar@<3pt>[dd]^-{F} \\ & \ar@{=>}[dr]^-{\phi}& \\ \B_{S'} & & \B_{S}\ar@<2pt>[ll]^-{f^*}\\ S'\ar[rr]^-f && S.  }}

To say $\phi$ is \emph{monoidal} means that if for every pair $a,b$ of objects in $A$ the following diagram commutes:

\[\xymatrix{ F'(a\otimes b) \ar[r]^-{\phi_{a\otimes b}} \ar[d]_-{\sigma'_{a,b}} & f^*F(a\otimes b) \ar[d]_-{f^*\sigma_{a,b}} \\ F'(a)\otimes_{S'}F'(b) \ar[r]_-{\phi_a\otimes_{S'}\phi_b} & f^*F(a)\otimes_{S'}f^*F(b),  }\]

and so does

\[ \xymatrix{u_{S'}\ar[d]_-{\sigma'_0} \ar@{=}[r] & u_{S'} \ar[d]^-{f^*\sigma_0} \\ F'(u) \ar[r]_-{\phi_u} & f^*F(u).}\]

Here $\sigma_{a,b}$ and $\sigma_0$ are the isomorphisms that come with the monoidal functor $F$; similarly $\sigma'_{a,b}$ and $\sigma'_0$ are for $F'$.  To say $\phi$ is symmetric means that it satisfies some further conditions in the form of commutative diagrams \cite[page~257]{maclane_categories}.

\subsubsection{}\label{}  To define compositions in $\uline A$, suppose we have morphisms $(f,\phi):(S',F')\rightarrow (S,F)$ and $(f',\phi'): (S'',F'')\rightarrow (S',F')$.  We define \[(f,\phi)\circ (f',\phi')=(f\circ f',f'^*(\phi)\circ \phi'):(S'',F'')\rightarrow (S,F) .\]

\centerline{\xymatrix{&& A\ar[d]^-F \ar[dl]^-{F'} \ar[dll]_-{F''} \\ \B_{S''} & \B_{S'}\ar[l]^-{f'^*} & \B_{S}.\ar[l]^-{f^*} \\ S''\ar[r]^-{f'} & S' \ar[r]^-f & S}}

This way $\uline A$ is a category with a covariant functor $p: \uline A\rightarrow \mathfrak S$ sending \[(S,F)\mapsto S\] and \[(f,\phi)\mapsto f.\]  The category $\uline A$ should be thought of as the category of ``representations'' of the abelian tensor category $A$ with values in $\B$.

\begin{lem}  The functor $p: \uline A\rightarrow \mathfrak S$ defined above is a category fibred in groupoids.  \end{lem}

\begin{proof}  Given any morphism $f: S'\rightarrow S$ in $\mathfrak S$ and $(S,F)\in \uline A$, the object $(S',f^*\circ F)$ in $\uline A$ admits a morphism \[(f,\mathrm{Id}_{f^*\circ F}): (S',f^*\circ F)\rightarrow (S,F).\]This is a lifting of $f$.

Now suppose \[(f,\phi_1): (S',F_1)\rightarrow (S,F)\] is another lifting of $f$.  Then we have a commutative diagram

\[\xymatrix{(S',f^*\circ F) \ar[drr]^-{(f,\mathrm{Id}_{f^*\circ F})}\\ && (S,F) \\ (S',F_1) \ar[urr]_-{(f,\phi_1)} \ar@{-->}[uu]^-{(\mathrm{Id}_{S'}, \phi_1)} }\]

where the dashed arrow is the unique morphism making the diagram commutative and has image under $p$ equal to $\mathrm{Id}_{S'}$.  \end{proof}

\subsubsection{}  \emph{Remark.}  We can define a larger category $\uline{A}^+$ with the same objects as $\uline{A}$ but morphisms $(f,\phi)$ where $\phi$ is a symmetric monoidal natural \emph{transformation} between monoidal functors.  Essentially the same proof above shows that $\uline{A}^+\rightarrow \mathfrak S$ is a fibred category:  The morphism $(f,\mathrm{Id}_{f^*\circ F})$ is easily seen to be strongly cartesian and is a lifting of morphism $f:S'\rightarrow S$ with a given target $(S,F)$ in $\uline{A}^+_S$.

\subsubsection{}\label{par:funB}  Suppose we have a morphism between categories fibred in abelian tensor categories:

\[\xymatrix{\B \ar[d] \ar[r]^-q & \B' \ar[d] \\ \mathfrak S \ar[r]^-{q_0} & \mathfrak S'.}\]

Let $A$ be an abelian tensor category, then we have an induced morphism \[q_*:\uline{A}(\B)\rightarrow \uline{A}(\B')\] between categories fibred in groupoids over the functor $q_0$ defined by sending \[q_*:(S,F)\mapsto (q_0(S),q_S\circ F),\] where $q_S: \B_S\rightarrow \B'_{q_0(S)}$ is the restriction of $q$ to the fibre category $\B_S$.

A special case of this is when $\B$ is the fibre product $\mathfrak S\times_{\mathfrak S'}\B'=q_0^{-1}\B'$, then we also have \[\uline{A}(\B)\cong \mathfrak S\times_{\mathfrak S'}\uline{A}(\B')=q_0^{-1}\uline{A}(\B').\]

\subsubsection{}\label{}  Fix a category $\B\rightarrow \mathfrak S$ fibred in abelian tensor categories.  If $g:A\rightarrow A_1$ is a tensor functor between abelian tensor categories, then we have an induced morphism between categories fibred in groupoids \[g^*: \uline{A_1}\longrightarrow \uline {A}.\]

More explicitly, $g^*$ sends $(S,F_1)\mapsto (S, F_1 \circ g)$, and if $(f,\phi):(S',F'_1)\rightarrow (S,F_1)$ is a morphism in $\uline A$ then $g^*$ sends $(f,\phi)\mapsto (f, g^*\phi)$ where $g^*\phi: F'_1\circ g\rightarrow f^*\circ F_1\circ g$ is given by \[(g^*\phi)_a=\phi_{g(a)}: F'_1g(a)\longrightarrow f^*F_1g(a)\] for $a\in A$.

\subsubsection{}\label{}  Let $k$ be a \emph{perfect} field and let $A_1$ and $A_2$ be two abelian \emph{$k$-linear} tensor categories with unit objects $u_1$ and $u_2$, and tensor products $\otimes_1$ and $\otimes_2$ respectively.  Assume moreover that every object in $A_i$, $i=1,2$ is of finite length and every $\mathrm{Hom}$ space has finite dimension over $k$.

Then by \cite[5.13(i)]{deligne_tannakian}, the tensor product $A=A_1\otimes A_2$ exists and is a $k$-linear abelian category satisfying the same finiteness conditions above.%(\footnote{The existence of tensor product seems to have been generlized: \url{http://mathoverflow.net/questions/23278/tensor-product-of-abelian-categories}.})

This category admits a functor \[\boxtimes: A_1\times A_2\longrightarrow A=A_1\otimes A_2\] that is exact in each variable and satisfies \[\mathrm{Hom}_A(a_1\boxtimes a_2,b_1\boxtimes b_2)\cong\mathrm{Hom}_{A_1}(a_1,b_1)\otimes_\Lambda \mathrm{Hom}_{A_2}(a_2,b_2)\]for every $a_i, b_i\in A_i$.

By \cite[5.16]{deligne_tannakian} the category $A$ is moreover an abelian \emph{tensor} category with a tensor product $\otimes$ right exact in each variable and satisfying \begin{equation}\label{eq:tensor}(a_1\boxtimes a_2)\otimes (b_1\boxtimes b_2)\cong (a_1\otimes_1b_1)\boxtimes (a_2\otimes_2b_2).\end{equation}

Its unit object is $u=u_1\boxtimes u_2$.

In particular we have tensor functors $q_i:A_i\rightarrow A$ defined by $q_1:a_1\mapsto a_1\boxtimes u_2$ and $q_2:a_2\mapsto u_1\boxtimes a_2$.  Moreover \eqref{eq:tensor} implies \begin{equation}\label{eq:tensor1} a_1\boxtimes a_2\cong q_1(a_1)\otimes q_2(a_2)\end{equation}in $A$.

\begin{lem}\label{lem:prod}  Let $k$ be a perfect field, and let $A_1$ and $A_2$ be abelian $k$-linear tensor categories satisfying the finiteness conditions above, and let $A=A_1\otimes A_2$.  Let $\B\rightarrow \mathfrak S$ be a category fibred in abelian tensor categories.  Then we have an isomorphism of categories fibred in groupoids over $\mathfrak S$: \[Q:\uline{A}\longrightarrow \uline{A_1}\times_\mathfrak S\uline{A_2}.\]  \end{lem}

\begin{proof}  The functor $Q$ is given by the universal property of the fibre product $\uline{A_1}\times_\mathfrak S\uline{A_2}$ along with the functors $q_1^*: \uline{A}\rightarrow \uline{A_1}$ and $q_2^*: \uline{A}\rightarrow \uline{A_2}$.

More precisely, \[Q:(S,F)\mapsto ((S,F_1),(S,F_2),\mathrm{Id}_S)\] where $F_i: A_i\rightarrow \B_S$ sends $a_i\mapsto F(q_i(a_i))$.  It is straightforward to show that $F_i$ are tensor functors. %this involved choose the isomorphisms $\sigma$ etc.

Let $(f,\phi): (S',F')\rightarrow (S,F)$ be a morphism in $\uline A$ then $Q$ sends \[(f,\phi)\mapsto ((f,\phi_1),(f,\phi_2))\] where $\phi_i: F'_i\rightarrow f^*\circ F_i$ is give by \[\phi_{i,a_i}=\phi_{q_i(a_i)}: F'_i(a_i)=F(q_i(a_i))\longrightarrow f^*F(q_i(a_i))=f^*F_i(a_i).\]

Since $\phi$ is a \emph{symmetric monoidal} natural isomorphism, so are $\phi_i$.

We define a quasi-inverse $Q'$ as follows:  Let $((S,G_1),(S',G_2),g)$ be an object in $\uline{A_1}\times_\mathfrak S\uline{A_2}$ over $S$ where $g:S\rightarrow S'$ is an \emph{isomorphism}.  Then $Q'$ sends it to $(S,G)\in \uline{A}$ where \[G(a_1\boxtimes a_2)=G_1(a_1)\otimes_Sg^*G_2(a_2)\in\B_S.\]

Then we have \[Q'\circ Q:(S,F)\mapsto (S,F')\] where \[F'(a_1\boxtimes a_2)=F_1(a_1)\otimes_SF_2(a_2)=F(q_1(a_1))\otimes_SF(q_2(a_2))\cong F(a_1\boxtimes a_2)\] by \eqref{eq:tensor1} since $F$ is a tensor functor.  This implies that $Q'\circ Q\cong \mathrm{Id}$.

Conversely, we have \[Q\circ Q': ((S,G_1), (S',G_2),g)\mapsto ((S,G_1'),(S,G_2'),\mathrm{Id}_S)\] where \[\phi_{a_1}:G_1'(a_1)\cong G_1(a_1)\]and \[\psi_{a_2}:G_2'(a_2)\cong g^*G_2(a_2).\]

These isomorphisms are natural isomorphisms.

It is straightforward to verify that we have an isomorphism \[((\mathrm{Id}_S,\phi),(g,\psi)): ((S,G_1'),(S,G_2'),\mathrm{Id}_S)\stackrel{\cong}{\longrightarrow} ((S,G_1), (S',G_2),g).\]

This shows $Q\circ Q'\cong\mathrm{Id}$ and so we are done.  \end{proof}

\subsubsection{}\label{}  The basic example of the tensor product of two abelian categories is as follows.  Let $R_1$ and $R_2$ be two (not necessarily commutative) $k$-algebras which are right coherent; this means that every finitely generated right ideal is finitely presented.

Denote by $A_i$ the abelian category of finitely presented right $R_i$-modules, and denote by $A$ the abelian category of finitely presented right $(R_1\otimes_{k}R_2)$-modules.  Then by \cite[5.3]{deligne_tannakian} the tensor product over $k$ \[\boxtimes=\otimes_{k}:A_1\times A_2\rightarrow A\] makes $A$ the tensor product of $A_1$ with $A_2$.  In this case we only need to assume $k$ to be a commutative ring.

As an example:

\begin{prop}\label{prop:prodgroup}  Let $\B\rightarrow \mathfrak S$ be a category fibred in abelian tensor categories, and let $H_1$ and $H_2$ be finite groups.  Let $k$ be a field, then we have an isomorphism of categories fibred in groupoids over $\mathfrak S$:  \[\uline{(H_1\times H_2)\mathrm{-rep}_k}\stackrel{\cong}{\longrightarrow} \uline{H_1\mathrm{-rep}_k}\times_\mathfrak S\uline{H_2\mathrm{-rep}_k}.\]  \end{prop}

\subsection{Stack and descent}\label{sec:stack}

\subsubsection{}\label{}  Now let $\mathfrak S$ be a site.  Let $\pi: \B\rightarrow \mathfrak S$ be a category fibred in abelian tensor categories.

Let $S$ be an object in $\mathfrak S$ and let $\mathcal U=\{u_i: T_i\rightarrow S\}$ be a covering.  Denote by $T_{ij}$ the product $T_i\times_ST_j$ and $p_\ell$ the $\ell$-th projection from $T_{ij}$.

We have canonical natural isomorphisms \[can_{ij}: p_1^*u_i^*\longrightarrow p_2^*u_j^*\] between functors from $\B_S$ to $\B_{T_{ij}}$, and this gives a functor $\delta$ from the category $\B_S$ to the category $DD(\mathcal U)$ of descent data with respect to $\mathcal U$.  More precisely, we have \[\delta: x\mapsto (u_i^*(x), can_{ij,x})\in DD(\mathcal U).\]

Notice that the category $DD(\mathcal U)$ is naturally an abelian tensor category, and the functor $\delta$ is a tensor functor.

\begin{thm}\label{thm:stack}  Let $\pi:\B\rightarrow \mathfrak S$ be a category fibred in abelian tensor categories over a category $\mathfrak S$ with a given Grothendieck topology.  Let $A$ be an abelian tensor category.
\begin{enumerate}[(i)]
\item If the functor $\delta$ is fully faithful for every $S\in\mathfrak S$ and covering $\mathcal U$ then $p: \uline A\rightarrow \mathfrak S$ is a prestack.
\item If the functor $\delta$ is fully faithful and essentially surjective for every $S\in\mathfrak S$ and covering $\mathcal U$, then $p:\uline A\rightarrow \mathfrak S$ is a stack.
\end{enumerate}  \end{thm}

This theorem will be proved in the next few paragraphs; its slogan is:  ``If $\B$ is a (pre)stack, then so is $\uline A$''.  We remark that the converse statements also hold.

The conditions hold in the cases of Example \ref{par:exqcoh} (the descent of quasi-coherent sheaves in the fpqc topology) and Example \ref{par:exqrep} (see \cite{liu_sq}).  On the other hand, in the case of finite groups Example \ref{par:exgrep} (with covering families given by collections of subgroups whose union is equal to $G$), only the conditions in (i) hold, and so we get only a prestack.

\subsubsection{}\label{}  \emph{Remarks.}  The novelty of Theorem \ref{thm:stack} is that we consider categories of \emph{tensor} functors and \emph{monoidal} natural isomorphisms between them.  An analogous result for arbitrary cartesian functors is a very special case of \cite[II, Corollaire~2.1.5]{giraud_cohnonab}, which states that if $\pi_A:\A\rightarrow \mathfrak S$ and $\pi_B:\B\rightarrow \mathfrak S$ are fibred categories over a site $\mathfrak S$, and $\B$ is a (pre)stack, then the fibred category \[\mathrm{CART}(\A,\B)\] over $\mathfrak S$ is also a (pre)stack.

Here the fibred category $\mathrm{CART}(\A,\B)$ is defined so that the fibre category $\mathrm{CART}(\A,\B)_S$ over $S\in\mathfrak S$ is the category \[\textbf{Cart}_{\mathfrak S_{/S}}(\A_{/S},\B_{/S})\] of cartesian functors between the fibred categories $\A_{/S}$ and $\B_{/S}$ over $S$.  Here $\A_{/S}$ is the category whose objects are pairs $(a,f)$ where $a\in \A$ and $f: \pi_A(a)\rightarrow S$ is a morphism in $\mathfrak S$.

To compare Giraud's theorem above with our situation, let $A$ be a category, and let $\A=A\times\mathfrak S$.  With the identity functor $\mathrm{Id}_A$ as the pull-back functor, $\A$ is a fibred category over $\mathfrak S$.  In this case we have $\A_{/S}\cong A\times (\mathfrak S_{/S})$.

Then we have an equivalence between fibre categories \[\mathrm{Fun}(A,\B_S)\longrightarrow \textbf{Cart}_{\mathfrak S_{/S}}(A\times (\mathfrak S_{/S}),\B_{/S})\] given by \[F\mapsto ((a,f)\mapsto (f^*F(a),f))\] with quasi-inverse \[\tilde F\mapsto \tilde F|_{A\times\mathrm{Id}_S}.\]

It seems possible to modify the proof of \cite[II, Corollaire~2.1.5]{giraud_cohnonab} to give the proof of a more general version of Theorem \ref{thm:stack}.  In the following we give a direct proof; by the remarks above, the main point is to show that tensor functors glue to tensor functors, and symmetric monoidal isomorphisms glue to symmetric monoidal isomorphisms.

\subsubsection{}\label{}  Let $x=(S_0,F)$ and $y=(S_0,G)$ be objects in $\uline A$ over the same object $S_0\in\mathfrak S$.  If $f:S\rightarrow S_0$ is a morphism, then we have objects $f^*x=(S,f^*\circ F)$ and $f^*y=(S,f^*\circ G)$ over $S$, and the usual definition \[\mathcal Isom_{\uline A}(x,y)(S)=\{ \phi: f^*\circ F\stackrel{\cong}{\longrightarrow}f^*\circ G \}\]
gives a presheaf of sets on the category $\mathfrak S_{S_0}$ of objects over $S_0$.  With $x$ and $y$ fixed and understood we simply denote this set by $\mathcal I(S)$.  For every $\phi\in \mathcal I(S)$ and morphism $h:a\rightarrow b$ in $A$, we have a commutative diagram \begin{equation}\label{eq:phi}\xymatrix{ f^*F(a) \ar[r]^-{\phi_a} \ar[d]_-{f^*F(h)} & f^*G(a)\ar[d]^-{f^*G(h)} \\ f^*F(b) \ar[r]^-{\phi_b} & f^*G(b)}\end{equation} in $\B_{S}$ whose rows are isomorphisms.

\begin{proof}[Proof of Theorem \ref{thm:stack}(i)]

Let $\mathcal U=\{u_i: T_i\rightarrow S\}$ be a covering in $\mathfrak S$, then we have the functor \[\delta: \B_{S}\longrightarrow DD(\mathcal U).\]

Let $a$ be an object in $A$ and let $\phi$ be in $\mathcal I(S)$; then $\phi_a$ is a morphism in $\B_{S}$ from $f^*F(a)$ to $f^*G(a)$.  By the faithfulness assumption, we see that $\phi_a$ is determined by its pull-backs \[u_i^*\phi_a: u_i^*f^*F(a)\longrightarrow u_i^*f^*G(a)\] in $\B_{T_i}$ since $\{u_i^*\phi_a\}=\delta(\phi_a)$.  Hence we conclude that the natural map \[\mathcal I(S)\longrightarrow \prod_i\mathcal I(T_i)\] induced by $u_i$ is an injection.

Now consider the next natural maps \[\prod_i\mathcal I(T_i)\rightrightarrows \prod_{i,j}\mathcal I(T_{ij})\] induced by the two projections $p_1,p_2$ from $T_{ij}$.  Suppose \[\phi'_i: u_i^*\circ f^*\circ F\longrightarrow u_i^*\circ f^*\circ G\] are in $\prod \mathcal I(T_i)$ with the same image under the two arrows into $\prod \mathcal I(T_{ij})$.  To be precise, this means that $\phi'_i$ are monoidal natural isomorphisms such that the diagram of functors from $A$ into $\B_{T_{ij}}$:

\[\xymatrix{ p_1^*u_i^*f^*F\ar[rr]^-{p_1^*\phi'_{i}} \ar[d]_-{can_{ij}} && p_1^*u_i^*f^*G\ar[d]^-{can_{ij}}  \\  p_2^*u_j^*f^*F\ar[rr]_-{p_2^*\phi'_{j}} && p_2^*u_j^*f^*G}\]is commutative.  (Recall that $can_{ij}$ are the canonical natural isomorphism between $p_1^*u_i^*$ and $p_2^*u_j^*$.)

Then for any $a\in A$ we get a isomorphisms \[\phi'_{i,a}:u_i^*f^*F(a)\longrightarrow u_i^*f^*G(a),\]and the commutative diagram above means precisely that $\{\phi'_{i,a}\}$ is a morphism in $DD(\mathcal U)$ from $(u_i^*f^*F(a),can_{ij,f^*F(a)})$ to $(u_i^*f^*G(a),can_{ij,f^*G(a)})$.

By the fullness assumption we get a (unique) isomorphism $\phi_a$ from $f^*F(a)$ to $f^*G(a)$ such that $u_i^*\phi_a=\phi'_{i,a}$ for every $i$.

The association $a\mapsto \phi_a$ is a natural transformation:  To this end we need to show that the diagram \eqref{eq:phi} is commutative for every morphism $h:a\rightarrow b$ in $A$.  But this diagram is commutative when applied with $u_i^*$ since $\phi'_i$ is a natural transformation, for every $i$, hence we conclude that $\phi$ is a natural transformation by the faithfulness of $\delta$.

Finally, we need to verify that $\phi$ is a \emph{symmetric monoidal} natural isomorphism.  We verify that it is monoidal, while its being symmetric can be shown in the same way.  That is, we need to verify that the following diagrams

\[\xymatrix{ f^*F(a\otimes b) \ar[r]^-{\phi_{a\otimes b}} \ar[d]_-{f^*\sigma_{a,b}} & f^*G(a\otimes b) \ar[d]_-{f^*\tau_{a,b}} \\ f^*F(a)\otimes_{S}f^*F(b) \ar[r]_-{\phi_a\otimes_{S}\phi_b} & f^*G(a)\otimes_{S}f^*F(b)  }\]and

\[ \xymatrix{u_{S}\ar[d]_-{f^*\sigma_0} \ar@{=}[r] & u_{S} \ar[d]^-{f^*\tau_0} \\ f^*F(u) \ar[r]_-{\phi_u} & f^*G(u)}\]are commutative.  But this follows once again from the faithfulness assumption and the fact that these diagrams are commutative when applied with $u_i^*$ for every $i$, since $\phi'_i$ are monoidal transformations.

Hence we conclude that the sequence \[\mathcal I(S)\rightarrow \prod_i\mathcal I(T_i)\rightrightarrows \prod_{i,j}\mathcal I(T_{ij})\] is exact, as required.  \end{proof}

\subsubsection{}\label{}  Here we finish the proof of Theorem \ref{thm:stack}.

\begin{proof}[Proof of Theorem \ref{thm:stack}(ii)]  Let $\mathcal U=\{u_i:T_i\rightarrow S\}$ be a covering in $\mathfrak S$.  Suppose we have objects $(T_i,F_i)$ in $\uline A$ along with isomorphisms $\phi_{ij}:p_1^*\circ F_i\rightarrow p_2^*\circ F_j$ satisfying the cocycle condition \[p^*_{13}\phi_{ik}=(p^*_{23}\phi_{jk})\circ (p^*_{12}\phi_{ij})\]  for every $i,j,k$, where $p_{\ell,\ell'}$ are projections from $T_{ijk}=T_i\times_ST_j\times_ST_k$.

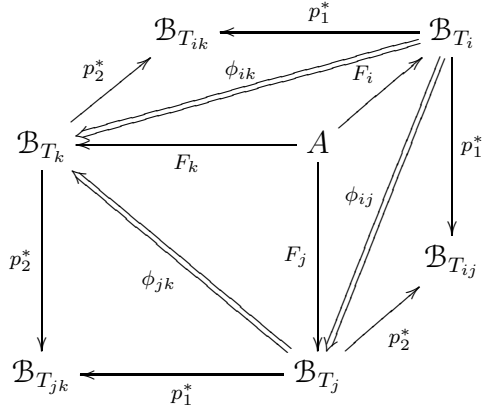
\[\xymatrix{ & \B_{T_{ik}} && \B_{T_i}\ar[dd]^-{p^*_1} \ar[ll]_-{p^*_1} \ar@{=>}[dddl]_-{\phi_{ij}} \ar@{=>}[dlll]_-{\phi_{ik}}\\  \B_{T_k}  \ar[ur]^-{p^*_2} \ar[dd]_-{p^*_2}&& A\ar[ll]^-{F_k} \ar[dd]_-{F_j}  \ar[ur]^-{F_i}\\  &&& \B_{T_{ij}} \\ \B_{T_{jk}} && \B_{T_j} \ar[ur]_-{p^*_2} \ar[ll]^-{p^*_1} \ar@{=>}[uull]^-{\phi_{jk}} }\]

Then for every object $a$ in $A$, we have isomorphisms $\phi_{ij,a}: p_1^*F_i(a)\rightarrow p_2^*F_j(a)$ satisfying the cocyle condition.  That is, $(F_i(a),\phi_{ij,a})$ is an object in $DD(\mathcal U)$.

Hence by the essential surjectivity assumption we have an object denoted suggestively as $F(a)$ in $\B_{S}$ along with isomorphisms $\lambda_i(a): u_i^*F(a)\rightarrow F_i(a)$ satisfying

\begin{equation}\label{eq:comp}\xymatrix{ p_1^*u_i^*F(a) \ar[r]^-{p_1^*\lambda_i(a)} \ar[d]_-{can_{ij,F(a)}} & p_1^*F_i(a)\ar[d]^-{\phi_{ij,a}} \\   p_2^*u_j^*F(a) \ar[r]^-{p_2^*\lambda_j(a)} & p_2^*F_j(a).   }\end{equation}

%(\footnote{Reference: for example in Stack Book, Chapter 4, Section 2.  Also for the cocycle conditions above.})

If $h: a\rightarrow b$ is a morphism in $A$, then we define $F(h): F(a)\rightarrow F(b)$ to be the unique morphism pulling-back via $u_i$ (for each $i$) to the composition \[u_i^*F(a)\stackrel{\lambda_i(a)}{\longrightarrow} F_i(a)\stackrel{F_i(h)}{\longrightarrow} F_i(b) \stackrel{\lambda_i(b)^{-1}}{\longrightarrow} u_i^*F(b).\]

To be more careful, we need to show that the composition above indeed is a morphism in $DD(\mathcal U)$.  That is, we need to verify that the following diagram is commutative:

\begin{equation}\label{eq:glue}  \xymatrix{p_1^*u_i^*F(a) \ar[r]^-{p_1^*\lambda_i(a)}\ar[d]_-{can_{ij,F(a)}} & p_1^*F_i(a)\ar[r]^-{p_1^*F_i(h)} \ar[d]_-{\phi_{ij,a}} & p_1^*F_i(b) \ar[r]^-{p_1^*\lambda_i(b)^{-1}} \ar[d]^-{\phi_{ij,b}} & p_1^*u_i^* F(b)\ar[d]^-{can_{ij,F(b)}} \\  p_2^*u_j^*F(a) \ar[r]_-{p_2^*\lambda_j(a)} & p_2^*F_j(a)\ar[r]_-{p_2^*F_j(h)} & p_2^*F_j(b) \ar[r]_-{p_2^*\lambda_j(b)^{-1}} & p_2^*u_j^* F(b).} \end{equation}

The commutativity of the first and the last squares is the compatibility condition \eqref{eq:comp}, and the commutativity of the second square follows from the fact that $\phi_{ij}$ is a natural transformation.

The association $a\mapsto F(a)$ from $A$ to $\B_{S}$ is then a functor: the fact that it respects composition of morphisms follows from the faithfulness assumption and that $F_i$ are functors.  Moreover, using the faithfulness and fullness assumptions it is straightforward to show that $F$ preserves colimits since the $F_i$ do.

It remains to show that $F$ is a symmetric monoidal functor, in particular we need to define for every $a,b\in A$ an isomorphism \[\sigma_{a,b}: F(a\otimes b)\longrightarrow F(a)\otimes_S F(b)\] in $\B_{S}$ as well as \[\sigma_0: u_S\longrightarrow F(u).\]

We have isomorphisms in $\B_{T_i}$: \[u_i^*F(a\otimes b)\stackrel{\lambda_i(a\otimes b)}{\longrightarrow} F_i(a\otimes b)\stackrel{\sigma_{i,a,b}}{\longrightarrow} F_i(a)\otimes_{T_i}F_i(b)\longrightarrow u_i^*F(a)\otimes_{T_i}u_i^*F(b),\]where $\sigma_{i,a,b}$ is the isomorphism for the tensor functor $F_i$, and the last arrow is the inverse of $\lambda_i(a)\otimes_{T_i}\lambda_i(b)$.

By considering a diagram similar to \eqref{eq:glue} we see that there is a unique isomorphism \[\sigma_{a,b}: F(a\otimes b)\rightarrow F(a)\otimes_SF(b)\] for every $a,b\in A$ such that $u_i^*\sigma_{a,b}=\sigma_{i,a,b}$.  (Here we used the fact that the isomorphisms $\phi_{ij}$ are \emph{monoidal} natural isomorphisms.)

The isomorphisms $\sigma_{a,b}$ are required to satisfy coherence conditions in the form of commutativity diagrams.  These follow from the commutativity of diagrams after pulling-back via $u_i$ using the faithfulness assumption.  The isomorphism $\sigma_0$ can be constructed in the same way.  Details are omitted.

Hence we conclude that $(S,F)$ is an object in $\uline A$ over $S$ pulling-back to $F_i$ via $u_i^*$, as required.  \end{proof}

\subsection{Sheaves of modules}

\subsubsection{}\label{}  In this section we work under the conditions of Theorem \ref{thm:stack}.  Namely, $\B\rightarrow \mathfrak S$ is a category fibred in abelian tensor categories over a site $\mathfrak S$ where all the functors $\delta: \B_S\rightarrow DD(\mathcal U)$ for all covering $\mathcal U$ of $S\in\mathfrak S$ are equivalences.

In this case we define a sheaf of rings on $\mathfrak S$ as follows.  Recall that every fibre category $\B_S$ has a unit object $u_S$.  Let $\mathcal O$ be the presheaf on $\mathfrak S$ given by \[S\mapsto \mathcal O(S):=\mathrm{End}_{\B_S}(u_S),\]where if $f: S'\rightarrow S$ is a morphism in $\mathfrak S$ then the restriction map $\mathcal O(S)\rightarrow \mathcal O(S')$ is induced by the functor $f^*: \B_S\rightarrow \B_{S'}$.

The assumption that the descent functors $\delta$ are fully faithful implies that this is indeed a sheaf.  This is a sheaf of \emph{commutative} rings if we impose the additional assumption on the categories $\B_S$ that the two isomorphisms $u_S\otimes_Sx\cong x$ and $x\otimes_Su_S\cong x$ for any $x\in\B_S$ give the same isomorphism $u_S\otimes_Su_S\cong u_S$ when $x$ is taken to be $u_S$ \cite[Lemma~9.6]{balmer_presheaves}.

\subsubsection{}\label{}  Let $A$ be an abelian tensor category, then we have a stack $p:\uline A\rightarrow \mathfrak S$ by Theorem \ref{thm:stack}.  Composing with the functor $p$ defines a ring-valued \emph{presheaf} on $\uline A$ which we will denote by $\mathcal O_{\uline A}$: \[(S,F)\mapsto \mathrm{End}_{\B_S}(u_S).\]

There is an inherited Grothendieck topology on the category $\uline A$:  A family of morphisms \[\mathcal V=\{v_i:y_i\rightarrow x\}\] with a fixed target $x\in\uline A$ is by definition a covering if $p(\mathcal V)=\{p(v_i):p(y_i)\rightarrow p(x)\}$ is a covering in $\mathfrak S$.  In this topology the presheaf defined above is a sheaf on $\uline{A}$.

We denote by $\Mod{\mathcal O_{\uline A}}$ the category of sheaves of right $\mathcal O_{\uline A}$-modules on $\uline{A}$.  With $\mathcal O_{\uline A}$ on $\uline A$ as the \emph{structure sheaf}, we have the category $\qcoh{\uline A}$ of quasi-coherent sheaves on $\uline{A}$.

%(\footnote{The fact that this is a Grothendieck topology on $\uline A$ is Definition 48.10.2 of Stack Project.})

%(\footnote{Definition 56.6.1 of Stack Project.})

%(\footnote{Definition 56.11.1, or rather Definition 15.23.1(5).})

%The category $\qcoh{\uline A}$ admits colimits and tensor product.  (\footnote{Lemma 58.11.8.  Remarks before this lemma says that the category is \emph{not} known to be abelian!!})

\subsubsection{}\label{par:sheaf}  Now we define a natural covariant functor \[\mathcal F:A\longrightarrow \Mod{\mathcal O_{\uline A}}.\]

For every $a\in A$ let $\mathcal F_a$ be the contravariant functor on $\uline{A}$ defined by \[\mathcal F_a: (S,F)\mapsto \mathrm{Hom}_{\B_S}(u_S,F(a)).\]

This is an object in $\Mod{\mathcal O_{\uline A}}$.  The functor $a\mapsto \mathcal F_a$ is additive but probably not left or right exact.  It is not clear under what conditions on the category $A$ is the functor $a\mapsto \mathcal F_a$ full, faithful, or a tensor functor.

\subsubsection{}\label{}  We say an object $a\in A$ is \emph{locally of finite presentation (with respect to $\pi:\B\rightarrow \mathfrak S$)} if for every $x=(S,F)\in\uline{A}$ there exists a covering $\mathcal V=\{v_i:y_i\rightarrow x\}$ where $y_i=(T_i,F_i)$ and $v_i=(u_i,\phi_i)$ such that for every $i$ the object $F_i(a)$ in $\B_{T_i}$ admits a finite presentation \begin{equation}\label{eq:fp}u_{T_i}^p\longrightarrow u_{T_i}^q\longrightarrow F_i(a)\longrightarrow 0.\end{equation}

More generally, we say $a$ is \emph{locally finite} if such a covering exists with a surjection \[u_{T_i}^q\longrightarrow F_i(a)\longrightarrow 0.\]

\begin{lem}  Suppose for every $S\in\mathfrak S$, that $u_S$ is a projective object in $\B_S$.
\begin{enumerate}[(i)]
\item The functor $a\mapsto \mathcal F_a$ is right exact.
\item Let $a\in A$ be locally of finite presentation, then the sheaf $\mathcal F_a$ on $\uline{A}$ is locally of finite presentation.
\end{enumerate}
\end{lem}

\begin{proof}  (i)  Clear.

(ii)  Choose the covering $\mathcal V$ for $a$ as in the definition above.  Applying the \emph{exact} functor $\mathrm{Hom}_{\B_{T_i}}(u_{T_i},-)$ to the sequence \eqref{eq:fp} gives the result.   \end{proof}

\subsubsection{}\label{}  \emph{Remarks.}  We should point out that we could replace $u_S$ in the constructions above with any system of objects $v_S\in \B_S$ admitting a morphism $f^*v_S\rightarrow v_{S'}$ in $\B_{S'}$ for every morphism $f: S'\rightarrow S$ in $\mathfrak S$.  This results in a different ``structure sheaf'', which may be more useful.  For instance, in the case when $\mathfrak S$ is the category of quivers, the vertex-wise tensor unit object carries less information than the path algebra.

Slightly more generally, consider a bi-functor $A\times A\rightarrow \Mod{\mathcal O_{\uline A}}$ defined by \[(a,b)\mapsto \mathcal F_{a,b}\] where \[\mathcal F_{a,b}((S,F))= \mathrm{Hom}_{\B_S}(F(a),F(b)).\]

Then the sheaves $\mathcal F_{a,b}$ on $\uline A$ have the additional structure of composition \[\circ: \mathcal F_{a,b}\times \mathcal F_{b,c}\rightarrow \mathcal F_{a,c}.\]

In particular every $\mathcal F_{a,a}$ is a sheaf of rings over which $\mathcal F_{a,b}$ is a right module and $\mathcal F_{b,a}$ is a left module.  Moreover, there is a natural morphism \[\mathcal F_{a,b}\otimes_{\mathcal O_{\uline A}} \mathcal F_{c,d}\rightarrow \mathcal F_{a\otimes c,b\otimes d}\]

\section{Example: Schemes}\label{sec:scheme}
In this Section we show that our construction in Section \ref{sec:ABstack} can be used to recover schemes from their category of quasi-coherent sheaves.

\subsection{Reconstruction of affine schemes}

\subsubsection{}\label{} Fix a commutative ring $\Lambda$ with identity, and all schemes considered in this section will be over $\Lambda$.

Let $\mathfrak S$ be the category of affine schemes (over $\Lambda$) with the \'etale topology, and $\B\rightarrow \mathfrak S$ as in Example \ref{par:exqcoh}; in particular for every affine scheme $S$ we have $\B_S=\qcoh{S}$.

\subsubsection{}\label{} Let $X$ be a scheme; denote by $\uline X$ the associated stack over $\mathfrak S$.  Consider the abelian tensor category $\qcoh{X}$ of quasi-coherent sheaves with its sheaf tensor product.  Then we have a morphism of stacks \[\uline X\stackrel{\alpha}{\longrightarrow} \uline{\qcoh{X}},\] sending $\alpha:(f: S\rightarrow X)\mapsto (S,f^*)$ and $\alpha:(g: S'\rightarrow S)\mapsto (g,\phi)\in\mathrm{Hom}((S',f'^*), (S,f^*))$ where $\phi$ denotes the natural isomorphism $f'^*\rightarrow g^*f^*$.

\subsubsection{}\label{}  \emph{Example.}  Suppose $\Lambda=k$ is a field.  Let $A$ be the abelian tensor category $\mathrm{Vect}_k$ of (possibly infinite dimensional) vector spaces over $k$ with the usual tensor product.  Then we have $A=\qcoh{\mathrm{Spec}(k)}$ and a stack morphism \[\uline{\mathrm{Spec}(k)}\longrightarrow \uline A.\]

If $(S,F)\in\uline A$ then we must have $F(k)\cong \mathcal O_S$, and so $\uline A{}_S$ has only one object (up to isomorphism) for every $S\in\mathfrak S$, since every object in $A$ is a direct sum of copies of $k$, and $F$ preserves direct sums.

Now suppose $(\mathrm{Id}_S,\phi)$ is an automorphism of $(S,F)$ in $\uline A$; in particular $\phi:F\rightarrow F$ is a \emph{monoidal} natural isomorphism.  Then $\phi_k$ is an automorphism of $\mathcal O_S$.  But then we have \[\phi_k=\phi_{k\otimes k}=\phi_k\otimes \phi_k=(\phi_k)^2\]as elements in $\mathrm{Aut}(\mathcal O_S)$, hence we must have $\phi_k=\mathrm{Id}_{\mathcal O_S}$.

Similarly we see that there is always a unique lifting $(f,\phi):(S',F')\rightarrow (S,F)$ in $\uline A$ for every scheme morphism $f:S'\rightarrow S$.  And we conclude that the morphism \[\alpha: \uline{\mathrm{Spec}(k)}\longrightarrow \uline A\] is an isomorphism.

\subsubsection{}\label{}  More generally, we have

\begin{thm}\label{thm:reconst}  Let $X$ be an affine scheme, then the morphism $\alpha: \uline X \rightarrow \uline{\qcoh{X}}$ is an isomorphism.  \end{thm}

For any tensor functor $F:A\rightarrow B$ between abelian tensor categories with unit objects $u_A$ and $u_B$ respectively, we denote by $F(1): \mathrm{End}(u_A)\rightarrow \mathrm{End}(F(u_A))\cong \mathrm{End}(u_B)$ the homomorphism given by $F$; here the last two endomorphism rings are identified using the isomorphism $\sigma_0: u_B\rightarrow F(u_A)$.

\begin{lem}\label{lem:ring}  Let $X$ be an affine scheme, and let $S$ be a scheme.  If $\psi: \Gamma(\mathcal O_X)\rightarrow \Gamma(\mathcal O_S)$ is a ring homomorphism inducing a scheme morphism $f:S\rightarrow X$, then the following diagram is commutative:

\centerline{\xymatrix{ \mathrm{End}(\mathcal O_X) \ar[r]^-{f^*(1)} \ar[d]_-\Gamma & \mathrm{End}(\mathcal O_S)\ar[d]^-\Gamma \\ \Gamma(\mathcal O_X) \ar[r]_-\psi & \Gamma(\mathcal O_S).  }}

Here $f^*:\qcoh{X}\rightarrow \qcoh{S}$ is the pull-back functor, and the vertical arrows are natural isomorphisms.\end{lem}

\begin{lem}\label{lem:right}  Let $X$ be an affine scheme, and let $S$ be a scheme.  Let $F,G: {\qcoh{X}}\rightarrow {\qcoh{S}}$ be two tensor functors.  If $F(1)=G(1)$ then $F\cong G$ via a symmetric monoidal natural isomorphism.  \end{lem}

\begin{proof}  Let $a\in\qcoh{X}$, then it admits a (possibly infinite) presentation \[\mathcal O_X^p\stackrel{m}{\longrightarrow} \mathcal O_X^q\longrightarrow a\longrightarrow 0,\]where $m$ is a matrix with entries in $\mathrm{End}(\mathcal O_X)$.  Applying the right exact functors $F$ and $G$ we get exact sequences

\[\xymatrix{  F(\mathcal O_X^p) \ar[rr]^-{F(1)(m)} \ar[d]_-\cong && F(\mathcal O_X^q) \ar[r]\ar[d]_\cong & F(a)\ar[r] & 0 \\ G(\mathcal O_X^p) \ar[rr]_-{G(1)(m)} && G(\mathcal O_X^q) \ar[r] & G(a) \ar[r] & 0,  }\]

where $F(1)(m)$ and $G(1)(m)$ are respectively the matrix with $F(1)$ and $G(1)$ applied to the entries of $m$, and the vertical isomorphisms are given by \[F(\mathcal O_X)\stackrel{\sigma_0^{-1}}{\longrightarrow} \mathcal O_S \stackrel{\tau_0}{\longrightarrow}  G(\mathcal O_X).\]  By the assumption $F(1)=G(1)$ the square in this diagram commutes, and so we have an isomorphism $\phi_{a}:F(a)\rightarrow G(a)$.

It is straightforward to verify that this is independent of the choice of the presentation of $a$ and moreover gives a natural isomorphism $\phi:F\rightarrow G$.  It remains to show that it is a \emph{symmetric monoidal} natural transformation.  We show that it is monoidal, leaving the symmetry to the reader.

So let $b\in\qcoh{X}$.  Consider the following diagram:

\[\xymatrix{
& F(b^q)\ar[ddl]_-{\phi_{b^q}} \ar[rr] \ar[d]^-{\sigma_{\mathcal O_X^q,b}}&& F(a\otimes b) \ar[ddl]_-(.2){\phi_{a\otimes b}}\ar[d]^-{\sigma_{a,b}}
\\ & F(b)^q \ar[rr]|(.73)\hole \ar[ddl]|\hole^-(.15){\phi_{\mathcal O_X^q}\otimes\phi_b} & &F(a)\otimes F(b) \ar[ddl]^-(.3){\phi_a\otimes\phi_b}
\\ G(b^q) \ar[rr] \ar[d]_-{\tau_{\mathcal O_X^q,b}}  && G(a\otimes b) \ar[d]_-{\tau_{a,b}}
\\ G(b)^q \ar[rr] & &G(a)\otimes G(b).   }\]

(Here $b^q$ stands for $\mathcal O_X^q\otimes b$, $F(b)^q$ stands for $\mathcal O_S^q\otimes F(b)$, etc.)

The commutativity of the right side is the condition we need to verify.

The top side is obtained by first applying $-\otimes b$ to the surjection $\mathcal O_X^q\rightarrow a$, and then the functors $F$ and $G$; the bottom side is obtained by applying these functor in the other order.  These two sides are commutative since $\phi$ is a natural transformation.

The front side and the back side are commutative by the compatibility conditions on the isomorphisms $\sigma$ and $\tau$.  Therefore we are reduced to prove the commutativity of the left side.

By considering a presentation $\mathcal O_X^s\rightarrow \mathcal O_X^t\rightarrow b\rightarrow 0$ of $b$, this is in turn reduced to the commutativity of the following diagram:

\[\xymatrix{  F(\mathcal O_X^t\otimes \mathcal O_X^q) \ar[rr]^-{\phi_{\mathcal O_X^t\otimes \mathcal O_X^q}} \ar[d]_-{\sigma_{\mathcal O_X^t,\mathcal O_X^q}}&&   G(\mathcal O_X^t\otimes \mathcal O_X^q) \ar[d]^-{\tau_{\mathcal O_X^t,\mathcal O_X^q}}
\\ F(\mathcal O_X^t)\otimes F(\mathcal O_X^q) \ar[rr]_-{\phi_{\mathcal O_X^t}\otimes \phi_{\mathcal O_X^q}} && G(\mathcal O_X^t)\otimes G(\mathcal O_X^q).}\]

It suffices to show that this diagram is commutative in the special case $q=t=1$.  Using the left diagram of (4) in \cite[page~256]{maclane_categories} (applied to both $F$ and $G$) we are reduced to showing the commutativity of the following diagram:

\[\xymatrix{ F(\mathcal O_X)\otimes \mathcal O_S \ar[rr]^{\phi_{\mathcal O_X}\otimes \mathrm{Id}} \ar[d]_-{\mathrm{Id}\otimes \sigma_0} && G(\mathcal O_X)\otimes \mathcal O_S   \ar[d]^-{\mathrm{Id}\otimes \tau_0}
\\ F(\mathcal O_X)\otimes F(\mathcal O_X) \ar[rr]^{\phi_{\mathcal O_X}\otimes \phi_{\mathcal O_X}} && G(\mathcal O_X)\otimes G(\mathcal O_X).}\]

But by construction we have $\phi_{\mathcal O_X}=\tau_0\circ \sigma_0^{-1}$, and so we are done.  \end{proof}

\begin{proof}[Proof of Theorem \ref{thm:reconst}]  We define a morphism \[\alpha': \uline{\qcoh{X}}\longrightarrow \uline X\] by sending $(S,F)$ to $f: S\rightarrow X$ where $f=\mathrm{Spec}(F(1))$.

We first show that $\alpha'\circ\alpha\cong\mathrm{Id}$.  So suppose $f: S\rightarrow X$ is an affine scheme over $X$.  Let $\psi=f^\#_X:\Gamma(\mathcal O_X)\rightarrow \Gamma(\mathcal O_S)$, then \[\mathrm{Spec}(f^*(1))=\mathrm{Spec}(\psi)=f,\]where we identified the rings using $\Gamma$ in Lemma \ref{lem:ring}.

Conversely, we need to show that $\alpha\circ \alpha'\cong\mathrm{Id}$.  So let $(S,F)$ be an object in $\uline{\qcoh{X}}$ over $S$.  Let $\psi=\Gamma(F(1)):\Gamma(\mathcal O_X)\rightarrow \Gamma(\mathcal O_S)$ as in Lemma \ref{lem:ring}, then we have \[F(1)=\mathrm{Spec}(\psi)^*(1),\]  hence we conclude $F\cong \mathrm{Spec}(\psi)^*$ by Lemma \ref{lem:right}.  \end{proof}

\subsubsection{}\label{par:rmkcoh}  \emph{Remarks.}  Recall from \ref{par:sheaf} that there is a functor from $\qcoh{X}$ into the category $\Mod{\mathcal O_X}$ on the stack $\uline{\qcoh{X}}\cong X$.  It is easily seen to be isomorphic to the inclusion functor.

The proof above also shows that we have an isomorphism $X\rightarrow \uline{A}$ whenever $A\subset \qcoh{X}$ is an abelian tensor subcategory; any such $A$ must contain the unit object $\mathcal O_X$ and hence every finitely presented objects in $\qcoh{X}$.  For example, when $X$ is a noetherian affine scheme, we may take $A$ to be $\coh{X}$.  (Here the noetherian assumption is used to guarantee that the tensor product of two coherent sheaves is still coherent.)

\subsubsection{}\label{par:ex}  \emph{Example.}  Here we show that it is crucial to use the sheaf tensor product in the proof of Theorem \ref{thm:reconst}.

Let $k$ be a field of characteristic not equal to $2$, and let $X=\mathrm{Spec}(k[t]/(t^2-1))$ be the affine scheme of two reduced points.  Let $A$ be the abelian category $\coh{X}$ on which we have the sheaf tensor product $\otimes_X$, then we have \[\uline X\cong \uline{(A,\otimes_X)}\] as stacks, as remarked in \ref{par:rmkcoh}.

There is, however, a different tensor product on the category $A$ by identifying the ring $k[t]/(t^2-1)$ with the group algebra $kG$, where $G=\{1,t\}$ is the cyclic group of order two.  This realizes $A$ as the category of finite dimensional $G$-representations; denote the representation tensor product on $A$ by $\otimes_G$.  Notice that this is indeed a different tensor product from $\otimes_X$ since they have different unit objects.

Every object $V$ in $A$ then decomposes as $V_+\oplus V_-$ where $V_+$ is the trivial representation and $t$ acts as the multiplication by $-1$ on $V_-$.

Denote by $k$ the one dimensional trivial $G$-representation, and let $M$ be the one dimensional $G$-representation on which $t$ acts as the multiplication by $-1$.  Then any tensor functor $F$ from $(A,\otimes_G)$ into an abelian tensor category $B$ is determined by the object $F(M)$, which must satisfy \[F(M)\otimes F(M)\cong F(k)\cong u_B.\]

And conversely, any object in $B$ whose tensor square is isomorphic to $u_B$ gives rise to such a functor.  Therefore we conclude that, if $S$ is an affine scheme, for instance, then the fibre category $\uline{(A,\otimes_G)}{}_S$ is isomorphic to the group of order two elements in $\mathrm{Pic}(S)$, namely, \[\uline{(A,\otimes_G)}{}_S\cong \mathrm{Pic}(S)[2].\]

In particular we have \[\uline{(A,\otimes_X)}\not\cong \uline{(A,\otimes_G)}.\]
See also Section \ref{sec:group} for a class of examples which illustrates the dependence on tensor structures.

\subsection{Reconstruction of schemes}

\subsubsection{}\label{par:gen}  Now we generalize Theorem \ref{thm:reconst} to more general schemes.

\begin{thm}\label{thm:reconst2}  Let $X$ be a quasi-compact and separated scheme, then the comparison morphism $\alpha: \uline{X}\longrightarrow \uline{\qcoh{X}}$ is an isomorphism.\end{thm}

The point is to construct a quasi-inverse \[\alpha': \uline{\qcoh{X}}\longrightarrow \uline{X},\]which means that for every affine scheme $S$ and a tensor functor $F:\qcoh{X}\rightarrow \qcoh{S}$, we need to define a scheme morphism $f:S\rightarrow X$.

We must do this locally on $S$ and $X$:  The idea is to glue open affine subschemes, but \emph{a priori} we do now even know if the stack $\uline{\qcoh{X}}$ is representable, and so, for example, if $U\subset X$ is an open affine subscheme, then it is not clear why the induced morphism \[U\cong \uline{\qcoh{U}}\longrightarrow \uline{\qcoh{X}}\] should be an \emph{open immersion}.

So let $a_1, a_2, \ldots$ be finitely presented quasi-coherent sheaves on $X$.  Then in particular each $a_i$ has a closed support.  Suppose that the intersection $\bigcap_i\mathrm{supp}(a_i)$ is empty on $X$.  Then the complements of $\mathrm{supp}(a_i)$ form an open covering of $X$.  Each $F(a_i)$ is finitely presented on $S$, and in particular each $\mathrm{supp}(F(a_i))$ is closed.

\begin{lem}\label{lem:cover}  With notations and assumptions as above, we have $\bigcap_i\mathrm{supp}(F(a_i))=\emptyset$ on $S$.  In particular the complements of $\mathrm{supp}(F(a_i))$ form an open covering of $S$.
\end{lem}

\begin{proof}  Suppose on the contrary that the intersection is non-empty on $S$.  Then there is a point $s\in S$ such that $s^*F(a_i)$ is non-zero in $\mathrm{Vect}_{k(s)}$ for every $i$.  Note that $s^*\circ F: \qcoh{X}\rightarrow \mathrm{Vect}_{k(s)}$ is a tensor functor, and so by Lemma \ref{lem:point} below, there is a point $x\in X$ such that $x^*(a_i)$ is non-zero for every $i$.  This means that the point $x$ lies in the intersection of $\mathrm{supp}(a_i)$, a contradiction.  \end{proof}

\begin{lem}\label{lem:point}  Let Let $X$ be a quasi-compact and separated scheme, and let $K$ be a field.  If $G:\qcoh{X}\rightarrow \mathrm{Vect}_K$ is a tensor functor, then there is a $K$-point $x$ on $X$ such that $G$ is isomorphic to $x^*$ as tensor functors.  \end{lem}

\begin{proof}  We first reduce to the affine case.  Let $U$ be an open affine subscheme of $X$, and let $M\subset \qcoh{X}$ be the kernel of the restriction functor $-|_U:\qcoh{X}\rightarrow \qcoh{U}$.  That is, $M$ is the full subcategory consisting of objects $b\in \qcoh{X}$ such that $b|_U\cong 0$; then we have \[M=\bigcap_{u\in U}\ker(u^*).\]

We claim that there is an open affine subscheme $U$ of $X$ such that $M$ is contained in $\ker(G)$.  Suppose this were not the case.  Then for every $x\in X$ there is an object $a_x\in\ker(x^*)$ not in $\ker(G)$.  We may choose $a_x$ to be of finite type:  Indeed, $a_x$ is a quasi-coherent sheaf and hence a direct limit of finite type subsheaves.  All of these finite type subsheaves must be in $\ker(x^*)$ in order to have $x^*(a_x)=0$, and at least one of these finite type subsheaves is not in $\ker(G)$ since otherwise we would have $G(a_x)\cong 0$, since $G$ commutes with direct limits.

So for every $x\in X$ we choose and fix a sheaf $a_x$ of finite type satisfying $a_x\in \ker(x^*)$ and $a_x\notin\ker(G)$.  Since $a_x$ is of finite type, its support $\mathrm{supp}(a_x)$ is closed in $X$; let $U_x=X-\mathrm{supp}(a_x)$.  Then $\{U_x\,|\, x\in X\}$ form an open covering of $X$.  Since $X$ is quasi-compact we have a finite sub-covering $U_{x_1}, U_{x_2}, \ldots, U_{x_n}$.  Let \[a:=\bigotimes_{i=1}^n a_{x_i}.\]  Then we have $x^*(a)=\bigotimes x^*(a_{x_i})=0$ since $x$ lies in one of $U_{x_i}$, for every $x\in X$.  This means that we must have $a\cong 0$ in $\qcoh{X}$.  But on the other hand, we have $G(a)\neq 0$ since every $G(a_{x_i})\neq 0$ in $\mathrm{Vect}_K$, a contradiction.

Therefore there is an open affine subscheme $U$ of $X$ such that $M=\ker(-|_U)$ is contained in $\ker(G)$.  Then we have a diagram \[\xymatrix{ \qcoh{X} \ar[d] \ar[r]^-G & \mathrm{Vect}_K\\ \qcoh{X}/M \ar@{-->}[ur]_-{\bar G}\ar[d]^-\cong & \\ \qcoh{U},  }\] where $\qcoh{X}/M$ is the localization \cite{gabriel}.  The lower vertical arrow is an equivalence by \cite[Chapter 3, proposition 5]{gabriel}.

More precisely, the equivalence follows from the fact that the push-forward functor induced by the open immersion $U\rightarrow X$ sends quasi-coherent sheaves on the \emph{affine} scheme $U$ to quasi-coherent sheaves on $X$, and this functor is a section functor of the restriction functor.  The fact that quasi-coherent sheaves are preserved under push-forward follows from our assumptions on the scheme $X$:  Indeed, we need the open immersion to be quasi-compact and separated \cite[Chapter 2, Proposition 5.8]{hartshorne}.  The separatedness follows from the fact that affine schemes are separated and \cite[Chapter 2, Corollary 4.6]{hartshorne}; the open immersion $U\rightarrow X$ is quasi-compact since for any open affine subscheme $Y$ of $X$, the intersection $Y\cap U$ is affine, and hence quasi-compact.

We then have a tensor functor $H: \qcoh{U}\rightarrow \mathrm{Vect}_K$ where $U$ is an open affine subscheme of $X$.  Thus we are reduced to the affine case.

The functor $H$ induces a ring homomorphism \[H(1):\Gamma(U,\mathcal O_U)\cong\mathrm{End}_{\qcoh{U}}(\mathcal O_U)\longrightarrow \mathrm{End}_{\mathrm{Vect}_K}(K)\cong K.\]  This gives a point $u:\mathrm{Spec}(K)\rightarrow U$.  Lemma \ref{lem:right} with $F=H$ and $G=u^*$ shows $H\cong u^*$.  \end{proof}

\subsubsection{}  Suppose now we choose the objects $a_i\in \qcoh{X}$ with sufficiently large support so that each $U_i:=X-\mathrm{supp}(a_i)$ is an open \emph{affine} subscheme.  The complement $V_i$ of $\mathrm{supp}(F(a_i))$ in $S$ is open but not necessarily affine.  We need to define a scheme morphism $f_i: V_i\longrightarrow U_i$, which will glue to give $f: S\rightarrow X$.

Since $U_i$ is affine, it suffices to give a ring homomorphism \[f_i^\#: \Gamma(U_i,\mathcal O_{U_i})\longrightarrow \Gamma(V_i,\mathcal O_{V_i}).\]  This can be constructed at the categorical level using the functor $F$ and localization as follows.

Consider the restriction tensor functor $-|_{U_i}:\qcoh{X}\rightarrow \qcoh{U_i}$; denote by $M_i=\ker(-|_{U_i})$ the subcategory of objects in $\qcoh{X}$ consisting of those objects $b$ satisfying $b|_{U_i}\cong 0$ in $\qcoh{U_i}$; or in other words object $b$ with $\mathrm{supp}(b)\subset \mathrm{supp}(a_i)$ as sets.

\begin{lem}  If $b\in K$ then $F(b)|_{V_i}\cong 0$ in $\qcoh{V_i}$.  \end{lem}

\begin{proof} It suffices to prove the following statement:  If $c,c'\in \qcoh{X}$ are such that $\mathrm{supp}(c')\subset \mathrm{supp}(c)$ as sets, then $\mathrm{supp}(F(c'))\subset \mathrm{supp}(F(c))$ as sets.  To see that this is enough, take $c=a_i$ and $c'=b$ in the situation above, this statement then implies $\mathrm{supp}(F(b))\subset \mathrm{supp}(F(a_i))=S-V_i$.

To prove the statement above, let $s\in S$ be a point in $\mathrm{supp}(F(c'))$.  Then $s^*F(c')\neq 0$.  By Lemma \ref{lem:point} there is a point $x\in X$ such that $s^*\circ F\cong x^*$, hence $x^*(c')\neq 0$.  In other words $x\in \mathrm{supp}(c')\subset \mathrm{supp}(c)$.  Therefore $0\neq x^*(c)\cong s^*F(c)$; that is, $s\in\mathrm{supp}(F(c))$.  \end{proof}

Then we have a diagram similar to the one in the proof of Lemma \ref{lem:point}:
\begin{equation}\label{eq:F}\xymatrix{ \qcoh{X} \ar[d] \ar[r]^-F & \qcoh{S} \ar[d] \\ \qcoh{X}/M_i \ar@{-->}[r]^-{F_i}\ar[d]^-\cong & \qcoh{V_i} \\ \qcoh{U_i}\ar@{-->}[ur]_-{\tilde F_i}.  }\end{equation}

The induced functor $F_i$ is a tensor functor, hence we have a ring homomorphism \[\Gamma(U_i,\mathcal O_{U_i})=\mathrm{End}(\mathcal O_{U_i})\stackrel{\cong}{\longrightarrow} \mathrm{End}_{\qcoh{X}/M_i}(\mathcal O_X)\stackrel{F_i}{\longrightarrow} \mathrm{End}(\mathcal O_{V_i})=\Gamma(V_i,\mathcal O_{V_i}),\] as desired.

\subsubsection{}  We show that the scheme morphisms $f_i: V_i\rightarrow U_i$ coincide on intersections $V_{ij}:=V_i\cap V_j$.  It then follows that there is a scheme morphism $f: S\rightarrow X$ for every given tensor functor $F: \qcoh{X}\rightarrow \qcoh{S}$.  This defines a morphism \[\alpha':\uline{\qcoh{X}}\longrightarrow \uline{X}.\]

Notice that $U_{ij}:=U_i\cap U_j$ is equal to $X-\mathrm{supp}(a_i)\cup \mathrm{supp}(a_j)=X-\mathrm{supp}(a_i\oplus a_j)$, and $V_{ij}=S-\mathrm{supp}(F(a_i\oplus a_j))$.  Since $X$ is quasi-compact and separated, and both $U_i$ and $U_j$ are affine, $U_{ij}$ is affine.  So by the same construction above we have a scheme morphism $g: V_{ij}\rightarrow U_{ij}$, and it only remains to show that this is equal to the restriction of $f_i: V_i\rightarrow U_i$.  This follows from the commutativity of the following diagram on the endomorphism rings of the identity objects:
\[\xymatrix{ \qcoh{X} \ar[d] \ar[r]^-F & \qcoh{S} \ar[d] \\ \qcoh{X}/M_i \ar@{-->}[r]\ar[d]^-\cong & \qcoh{V_i}\ar[dd] \\ \qcoh{U_i}\ar[d] \\ \qcoh{U_i}/M_{ij} \ar[d]^\cong \ar@{-->}[r] & \qcoh{V_{ij}}\\ \qcoh{U_{ij}}. }\]

\subsubsection{}  Now we can finish the

\begin{proof}[Proof of Theorem \ref{thm:reconst2}]  It suffices show that the functor $\alpha'$ defined above is a quasi-inverse to $\alpha$.

First we prove $\alpha\circ \alpha'\cong\mathrm{Id}$.  So suppose $F:\qcoh{X}\rightarrow \qcoh{S}$ is a tensor functor, which induces a scheme morphism $f: S\rightarrow X$ as above by choosing an open affine covering $X=\cup U_i$, $U_i=X-\mathrm{supp}(a_i)$.  Denote by $S=\cup V_i$ the corresponding open covering as in the construction above.  Let $h_i: U_i\hookrightarrow X$ and $g_i: V_i\hookrightarrow S$ be the open immersions, and let $f_i:V_i\rightarrow U_i$ be the restrictions of $f$.

We need to show that $(S,F)$ and $(S,f^*)$ are isomorphic objects in $\uline{\qcoh{X}}$.  Since this is a stack it suffices to show that their associated descent data with respect to the covering $\{h_i: V_i\rightarrow S\}$ are isomorphic.  That is, we need to show that $g_i^*\circ F$ and $g_i^*\circ f^*$ are isomorphic as tensor functors from $\qcoh{X}$ to $\qcoh{V_i}$.  (We also need to show that the isomorphisms we construct commute with the canonical natural transformations in the descent data; we leave this part to the reader.)

By the commutative diagram of schemes \[\xymatrix{V_i \ar[r]^-{f_i}\ar[d]_{g_i} & U_i \ar[d]^-{h_i} \\ S\ar[r]_-f & X}\] we have natural isomorphisms \[g_i^*\circ f^*\cong f_i^*\circ h_i^*.\]

On the other hand, denote by $\tilde F_i: \qcoh{U_i}\rightarrow \qcoh{V_i}$ the tensor functor constructed in diagram \eqref{eq:F}, we have \[g_i^*\circ F\cong \tilde F_i\circ h_i^*.\]  Hence it suffices to prove $f_i^*\cong \tilde F_i$.  But $\tilde F_i$ is constructed from $f_i$ and satisfies $\tilde F_i(1)=f_i^*(1)$, hence we conclude by Lemma \ref{lem:right}.

Conversely, we need to prove that $\alpha'\circ \alpha\cong\mathrm{Id}$.  So let $f: S\rightarrow X$ be a scheme morphism.  Cover $X$ with open affine subschemes $U_i$ of the form $X-\mathrm{supp}(a_i)$.  Let $V_i=f^{-1}(U_i)$ and let $f_i: V_i\rightarrow U_i$ be the restrictions of $f$.  Let $F=f^*: \qcoh{X}\rightarrow \qcoh{S}$, then we have as in diagram \eqref{eq:F} tensor functors \[F_i: \qcoh{X}/M_i\longrightarrow \qcoh{V_i}.\]

It suffices to prove that the scheme morphism $f'_i: V_i\rightarrow U_i$ induced by $F_i$ is equal to $f_i$.  In other words, we need to prove $F_i(1)=f_i^*(1)$ as ring homomorphisms from $\Gamma(U_i,\mathcal O_{U_i})\cong \mathrm{End}(\mathcal O_{U_i})$ to $\Gamma(V_i,\mathcal O_{V_i})\cong \mathrm{End}(\mathcal O_{V_i})$:  \[\xymatrix{ \qcoh{X}/M_i \ar@{-->}[r]^-{F_i}\ar[d]^-\cong & \qcoh{V_i} \\ \qcoh{U_i}\ar@{-->}[ur]_-{f_i^*}.  }\]

This follows from the fact that the functor $F_i$ is induced by the scheme morphism $f: S\rightarrow X$ which restricts to $f_i$.  \end{proof}

\subsubsection{}  \emph{Remark.}  The idea of considering supports of objects in a possibly abstract category is the starting point of tensor triangular geometry.  See \cite{balmer_presheaves} and \cite{balmer_ICM}.

\subsubsection{}\label{subsec:compare}  \emph{Remarks.} Theorem \ref{thm:reconst2} is a stronger version of a special case of Lurie's recontruction of geometric stacks (\cite[Theorem~5.11]{lurie} and \cite[Theorem~3.0.1]{DAGVIII}) which describes the essential image of the natural functor \[\mathrm{Hom}({S},{X})\stackrel{}{\longrightarrow}\mathrm{Fun}(\qcoh{X}, \qcoh{S})\] sending $f\mapsto f^*$.  More precisely, Lurie's theorem applies to geometric stack $X$ and states that this functor is fully faithful with essential image consisting of tensor functors which carry flat objects to flat objects.

In the case when $X$ is a quasi-compact separated scheme (which is an example of a geometric stack), Theorem \ref{thm:reconst2} states that this essential image consists of tensor functors.  Therefore every tensor functor is isomorphic (via a symmetric monoidal natural isomorphism) to a tensor functor which moreover carries flat objects to flat objects, and in this special case of Lurie's theorem we may remove the condition that the tensor functor preserves flat objects.

\section{Example: Finite group representations}\label{sec:group}
In this Section we consider the category $\rep{G}$ of finite dimensional linear representations of a finite group $G$. We show that our construction in Section \ref{sec:ABstack} applied to $\rep{G}$ equipped with tensor product of representations give the classifying stack $BG$. We also observe that a different tensor structure can be given to $\rep{G}$ and our construction produces a stack quite different from $BG$.

\subsection{The set of characters}\label{sec:char}

\subsubsection{}\label{}  Let $k$ be an algebraically closed field of characteristic zero, and let $G$ be a finite group.  The abelian category $\rep{G}$ of finite dimensional $G$-representations over $k$ is equivalent to the category $\modd{kG}$.

Let $Z(k G)$ be the center of the group algebra $k G$, then $Z(k G)$ is a commutative subalgebra.  It is isomorphic as an algebra to the direct sum $k^{\#\hat G}$ of $\#\hat G$ copies of $k$, where $\hat G$ is the set of characters on $G$.

We have a Morita equivalence: \[k G\mathrm{-mod} \stackrel{\cong}{\longrightarrow} G\mathrm{-rep} \stackrel{\chi}{\longrightarrow} Z(k G)\mathrm{-mod},\]
where the second arrow sends an irreducible representation to the one dimensional $k$-vector space spanned by its character.

More explicitly, denote the irreducible $G$-representation by $\rho_1, \ldots, \rho_{\#\hat G}$.  Then we have \[\chi: \rho\cong\bigoplus_i\rho_i^{\oplus m_i}\mapsto \prod_i k^{m_i}\] where the $k$-vector space on the right is a $Z(kG)$-module with the ``diagonal'' action.  More intrinsically, we have \[\rho\cong\bigoplus_i \mathrm{Hom}_G(\rho_i,\rho)\otimes_k\rho_i,\]  then we have \[\chi(\rho)=\prod_i \mathrm{Hom}_G(\rho_i,\rho).\]

If $\lambda: \rho_i\rightarrow \rho_i$ is the morphism given by multiplication by $\lambda\in k$, then \[\chi(\lambda): \chi(\rho_i)\rightarrow \chi(\rho_i) \] is the multiplication by $\lambda$.

A quasi-inverse of the equivalence $\chi$ is given by \[\chi^{-1}: \prod_i W_i\mapsto \bigoplus_i W_i\otimes_k\rho_i,\] where each $W_i$ is a finite dimensional $k$-vector space.

\subsubsection{}\label{par:GZ}  Denote by $A$ the abelian category $G\mathrm{-rep}\cong Z(k G)\mathrm{-mod}$.  Let $\otimes_G$ be the representation tensor product on $G\mathrm{-rep}$, and let $\otimes_Z$ be the module tensor product on $Z(k G)\mathrm{-mod}$.  Consider abelian tensor categories $A_G:=(A,\otimes_G)$ and $A_Z:=(A,\otimes_Z)$ with identical underlying abelian categories.

By Theorem \ref{thm:reconst} we have a stack isomorphism \[\uline{\mathrm{Spec}(k^{\#\hat G})}\cong \uline{A_Z}\] over the \'etale site $\mathfrak S$ of affine $k$-schemes.  So $\uline{A_Z}$ is the disjoint union of $\#\hat G$ points.

\subsection{The representation tensor product}

\subsubsection{}\label{}  Now we consider $A_G$.  Consider the stack $BG$ over $\mathfrak S$, whose objects are pairs $(S,\mathcal M)$ where $\mathcal M$ is a sheaf on the affine scheme $S$ (in its \'etale topology) which is a $G$-torsor.  We define a natural functor \[\beta:BG\longrightarrow \uline{A_G}\] as follows.

Let $(S,\mathcal M)$ be an object in $BG$, then the $G$-torsor gives an element in $\check H^1(S,G)$.  If \[\rho: G\rightarrow GL(V_\rho)\] is a representation, then we have an induced map \[\rho_*: \check H^1(S,G)\longrightarrow \check H^1(S,GL(V_\rho)).\]

The element $\rho_*(\mathcal M)$ is a $GL(V_\rho)$-torsor over $S$, which gives a flat vector bundle over $S$; we denote this vector bundle also by $\rho_*(\mathcal M)$.  Then we define $\beta$ by sending \[\beta: (S,\mathcal M)\mapsto (S,\beta(\mathcal M)),\]  where $\beta(\mathcal M): A_G\rightarrow \qcoh{S}$ sends $\rho\mapsto \rho_*(\mathcal M)$.

More explicitly, given $\mathcal M$ we can find a covering $\mathcal U=\{u_i:T_i\rightarrow S\}$ so that the torsor $\mathcal M$ is glued from the trivial torsors $\{G\times T_i\}$ via the transition elements $\{g_{ij}\}$ satisfying the cocycle condition, where each $g_{ij}$ is an element in $G$, and multiplication by $g_{ij}$ gives the isomorphisms \[g_{ij}: p_1^*(G\times T_i)\longrightarrow p_2^*(G\times T_j)\] over $T_{ij}=T_i\times_ST_j$.

For any $\rho\in G\mathrm{-rep}$, the vector bundle $\beta(\mathcal M)(\rho)=\rho_*(\mathcal M)$ is then glued from the trivial vector bundles $\{V_\rho\otimes_\Lambda \mathcal O_{T_i}\}$ from the transition elements $\{\rho(g_{ij})\}$.  From this description it is clear that the functor $\beta(\mathcal M):\rho\mapsto \rho_*(\mathcal M)$ is indeed a tensor functor from $A_G$ into $\qcoh{S}$.

\begin{lem}\label{lem:faithful}  The functor $\beta: BG\rightarrow \uline{A_G}$ is faithful.  \end{lem}

\begin{proof}  It suffices to show that if $\mathcal M$ is a $G$-torsor over $S$, then $\beta$ induces an injection from $\mathrm{Aut}(\mathcal M)$ to $\mathrm{Aut}(\beta(\mathcal M))$; the latter automorphism group consists of symmetric monoidal natural automorphisms of the functor $\beta(\mathcal M)$.

Let $\mathcal M$ be given by $\{g_{ij}\}$ as above, then any automorphism of $\mathcal M$ is given by $\{h_i\}$, $h_i\in G$, satisfying \[h_jg_{ij}=g_{ij}h_i,\]and the automorphism $\beta(\{h_i\})_\rho: \rho_*(\mathcal M)\rightarrow \rho_*(\mathcal M)$ is given by $\{\rho(h_i)\}$.  Therefore we conclude by taking $\rho$ to be any faithful representation.  \end{proof}

\subsubsection{}\label{}  Let $S$ be a connected affine scheme, and take the trivial $G$-torsor $\mathcal M=G\times S$ as an example; notice that we have $\mathrm{Aut}(\mathcal M)\cong G$.  Then all the transition elements $g_{ij}$ are the identity element in $G$, and for every $\rho\in G\mathrm{-rep}$ we have \[\rho_*(\mathcal M)=V_\rho\otimes_k \mathcal O_S.\]

We claim that the composition \[G\stackrel{\cong}{\longrightarrow}\mathrm{Aut}(\mathcal M)\stackrel{\beta}{\longrightarrow} \mathrm{Aut}(\beta(\mathcal M))\] is an isomorphism.  By Lemma \ref{lem:faithful} it only remains to prove that it is a surjection.

For every point $p\in S$ denote by $k(p)$ its residue field, which is then a field extension of $k$.  Consider the composition tensor functor \[p^*\circ\beta(\mathcal M): A_G\stackrel{\beta(\mathcal M)}{\longrightarrow }\qcoh{S} \stackrel{p^*}{\longrightarrow} \mathrm{Vect}_{k(p)};\] this is a fibre functor in the sense of \cite[1.9]{deligne_tannakian}.

Since $\mathcal M$ is the pull-back of the trivial $G$-torsor $\mathcal M_0\rightarrow \mathrm{Spec}(k)$ via the structural morphism $S\rightarrow \mathrm{Spec}(k)$, we have a commutative diagram

\[\xymatrix{ G \ar[r]^-\cong \ar[dr]_-\cong& \mathrm{Aut}(\mathcal M)\ar[d]^-{p^*} \ar[r] & \mathrm{Aut}(\beta(\mathcal M))\ar[d]^-{p^*} \\ & \mathrm{Aut}(\mathcal M_0) \ar[r]_-\cong & \mathrm{Aut}(p^*\circ \beta(\mathcal M)),  }\] where the lower horizontal arrow is an isomorphism by the classical Tannakian duality, or the high-powered \cite[1.12]{deligne_tannakian}.

Let $\phi\in\mathrm{Aut}(\beta(\mathcal M))$, then the association given by the right vertical arrow above \[p\mapsto p^*\phi\] gives a map $S\rightarrow G$, under which the preimage of every group element in $G$ is closed.  Since $S$ is connected, this map must be a constant map, and so $\phi$ lies in the image of $G$.

\begin{lem}\label{lem:full}  The functor $\beta: BG\rightarrow \uline{A_G}$ is full.  \end{lem}

\begin{proof}  We need to show that if $\mathcal M$ is a $G$-torsor over $S$, then $\beta$ induced a surjection from $\mathrm{Aut}(\mathcal M)$ to $\mathrm{Aut}(\beta(\mathcal M))$.  The case when $\mathcal M$ is the trivial torsor is treated above.

Fix a covering $\mathcal U=\{u_i:T_i\rightarrow S\}$ such that each $u_i^*\mathcal M$ is trivial.  Then since both $BG$ and $\uline{A_G}$ are stacks, we have a commutative diagram with exact rows:

\[\xymatrix{ \mathrm{Aut}(\mathcal M) \ar[r]\ar[d] &\displaystyle \prod_i\mathrm{Aut}(u_i^*\mathcal M)\ar[d] \ar@<.5ex>[r] \ar@<-.5ex>[r]& \displaystyle \prod_{i,j}\mathrm{Aut}(u_{ij}^*\mathcal M) \ar[d]
\\ \mathrm{Aut}(\beta(\mathcal M)) \ar[r] &\displaystyle \prod_i\mathrm{Aut}(u_i^*\circ\beta(\mathcal M)) \ar@<.5ex>[r] \ar@<-.5ex>[r]& \displaystyle \prod_{i,j}\mathrm{Aut}(u_{ij}^*\circ\beta(\mathcal M)).   }\]

Since the second and the third vertical arrows are isomorphisms, so is the first.  \end{proof}

\subsubsection{}\label{}  Now consider the essential image of the functor $\beta: BG\rightarrow \uline{A_G}$.  Let $(S,F)\in \uline{A_G}$, then by \cite[2.7]{deligne_tannakian} we know that $F(\rho)$ is a vector bundle of finite rank on $S$ for every affine scheme $S\in\mathfrak S$ and every $\rho\in G\mathrm{-rep}$.  Moreover, by specializing to a closed point as in the argument before Lemma \ref{lem:full} we see that the rank of $F(\rho)$ is equal to the dimension of $V_\rho$ using the fact that there is essentially only one fibre functor into the category of $k$-vector spaces, namely the forgetful functor \cite[1.10]{deligne_tannakian}.

Since there are only finitely many isomorphism classes of irreducible representation, we can find a covering $\mathcal U=\{u_i:T_i\rightarrow S\}$ such that each $u_i^*F(\rho)$ is a trivial vector bundle on $T_i$ for every $\rho\in G\mathrm{-rep}$.  In particular, this means that the functor $u_i^*\circ F: A_G\rightarrow \qcoh{T_i}$ is isomorphic to $\beta(\mathcal M_i)$ where $M_i=G\times T_i$ is the trivial $G$-torsor over $T_i$.

Denote by $h_{ij,\rho}$ the transition isomorphism $p_1^*u_i^*F(\rho)\rightarrow p_2^*u_j^*F(\rho)$ on $T_{ij}$ of the vector bundle $F(\rho)$.  This gives a natural isomorphism $h_{ij}: p_1^*u_i^*F\rightarrow p_2^*u_j^*F$ between functors from $A$ to $\qcoh{T_{ij}}$.  Therefore we have isomorphisms \[h_{ij}: p_1^*\beta(\mathcal M_i)\longrightarrow p_2^*\beta(\mathcal M_j).\]

Identifying these functors with $\beta(\mathcal M_{ij})$, where $\mathcal M_{ij}=G\times T_{ij}$ is the trivial $G$-torsor, we see that each $h_{ij}$ is an element in $\mathrm{Aut}(\beta(\mathcal M_{ij}))$, which is isomorphic to a product of copies of $G$ (the number of copies is equal to the number of connected components of $T_{ij}$).

Thus the datum $\{\beta(\mathcal M_i),h_{ij}\}$ glues in $\uline A$ to an object $\beta(\mathcal M)$ in the image of $\beta$, and the local isomorphisms $u_i^*F\cong \beta(\mathcal M_i)$ gives an isomorphism $F\cong \beta(\mathcal M)$.  Hence we conclude:

\begin{lem}\label{lem:esssurj}  The functor $\beta: BG\rightarrow \uline{A_G}$ is essentially surjective.  \end{lem}

Combining Lemmas \ref{lem:faithful}, \ref{lem:full}, and \ref{lem:esssurj} we have proven

\begin{thm}\label{thm:BG}  The functor $\beta: BG\rightarrow \uline{A_G}$ is an equivalence.\end{thm}

\subsubsection{}\label{}  \emph{Remark.}  In the case when $G$ is finite and \emph{abelian} the equivalence \[\beta: BG\rightarrow \uline{A_G}=\uline{(A,\otimes_G)}\] can be described more explicitly.

First notice that if $H_1$ and $H_2$ are finite groups, then we have \[B(H_1\times H_2)\cong BH_1\times_{\mathfrak S}BH_2.\]  Combining this with Proposition \ref{prop:prodgroup} and choosing any decomposition \[G\cong \prod_i (\mathbb Z/n_i\mathbb Z)\] we may reduce to the case when $G\cong\mathbb Z/n\mathbb Z\cong \mu_n$ is finite and cyclic.

Fix any embedding $\rho: G\rightarrow k^\times=GL(k)$, viewed as a one-dimensional representation.  Then $\{\rho, \rho^{\otimes 2}, \rho^{\otimes 3},\ldots, \rho^{\otimes n}\}$ is a full list of irreducible representations.  Therefore every tensor functor $F:G\mathrm{-rep}\rightarrow \qcoh{S}$ is determined by the line bundle $F(\rho)$ on $S$.

The line bundle $F(\rho)$ admits an isomorphism $F(\rho)^{\otimes n}\cong F(\rho^{\otimes n})\cong \mathcal O_S$.  Let \[m_F: F(\rho)\longrightarrow F(\rho)^{\otimes n}\cong \mathcal O_S\] be the $n$-th tensor power morphism.  Denote by $1\in \Gamma(S,\mathcal O_S)$ the nowhere vanishing global section.  Then the preimage $m_F^{-1}(1)$ is a subsheaf of $F(\rho)$ which is easily seen to be a $\mu_n$-torsor over $S$.  The association \[(S,F)\mapsto (S,m_F^{-1}(1))\] gives a quasi-inverse to the functor $\beta$.

\subsubsection{}\label{}  \emph{Remarks.}  (1)  Theorem \ref{thm:BG} above is very similar to Lurie's result applied to the geometric stack $X=BG$:  Indeed, recalling that the category $\Rep{G}$ of possibly infinite dimensional $G$-representations is equivalent to the category of $G$-equivariant sheaves on $\mathrm{Spec}(k)$, which in turn is equivalent to the category $\qcoh{BG}$.  Lurie's theorem \cite[Theorem~5.11]{lurie} states that the natural functor \[BG(S)\longrightarrow \mathrm{Fun}(\Rep{G}, \Mod{\mathcal O_S})\] has its essential image consisting of tensor functors which carry flat objects to flat objects.  Compare this with Remarks \ref{subsec:compare}.

(2) More generally, if $G$ acts on a scheme $W$ and we take $A$ to be the abelian tensor category of $G$-equivariant quasi-coherent sheaves on $W$, then $\uline{A}$ is isomorphic to the stack $[W/G]$.

(3)  If $G$ and $H$ are finite groups with the same number of conjugacy classes, then on the abelian category \[A=\rep{G}\cong \modd{Z(kG)}\cong \modd{Z(kH)}\cong \rep{H}\] there are tensor products $\otimes_G$, $\otimes_H$, and $\otimes_Z$ so that \[\uline{(A,\otimes_G)}\cong BG,\]\[\uline{(A,\otimes_H)}\cong BH,\] and \[\uline{(A,\otimes_Z)}\cong \hat G\cong \hat H.\]

\subsubsection{}\label{par:exsheafG}  By \ref{par:sheaf} there are functors from $A_G$ into the category of sheaves on $\uline{A_G}\cong BG$, and from $A_Z$ into the category of sheaves on $\uline{A_Z}\cong \hat{G}$.  Here we describe them with both underlying abelian categories $A_G$ and $A_Z$ realized as $\rep{G}$:

\[\xymatrix{ & Sh(\uline{A_Z}) &&& \mathcal G_\rho\\ \rep{G} \ar[ur]\ar[dr] &&& \rho \ar@{|->}[ur] \ar@{|->}[dr]\\ & Sh(\uline{A_G}) &&& \mathcal F_\rho   }\]

We identify $\uline{A_Z}=\hat G$ with the set $\{\chi_1, \chi_2, \ldots, \chi_{\#\hat G}\}$ of irreducible characters; denote by $\rho_i$ the irreducible representation with character $\chi_i$.  Then $\mathcal G_{\rho_i}$ corresponds to the one-dimensional $k^{\#\hat G}$-module supported at the point $\chi_i$.

On the other hand, let $\rho:G\rightarrow GL(V_\rho)$ be a representation.  The corresponding sheaf $\mathcal F_{\rho}$ on $\uline{A_G}\cong BG$ sends an object $(S,\mathcal M)\in BG$ to the global section of the vector bundle $\rho_*(\mathcal M)$ with fibres isomorphic to $V_\rho$ on $S$ corresponding to the $G$-torsor $\mathcal M$.  In particular, this sheaf restricts to the sheaf $\mathcal M$ on the category $BG_{/(S,\mathcal M)}\cong\mathfrak S_{/S}$ over $BG$.

\section{Example: the dual of $G$-gerbes}\label{sec:gerbe}
 In this Section we apply the framework of 2-descent of stacks recalled in Appendix \ref{sec:twisting} to realize $G$-gerbes and their duals as stacks which locally are of the form $\uline{A}$.

\subsection{Comparison}
This subsection contains some preparatory results, to be used in the main construction.
\subsubsection{}\label{par:comparison}  Consider the equivalences \[\modd{k G}\stackrel{\cong}{\longrightarrow} \rep{G}\stackrel{\cong}{\longrightarrow}\modd{Z(k G)},\]where $\modd{k G}$ is given the representation tensor $\otimes_G$, while $\modd{Z(k G)}$ is given $\otimes_Z$.  Denote the second equivalence by $\chi$, then it is \emph{not} a tensor functor: for instance, it does not send the $\otimes_G$-unit object to the $\otimes_Z$-unit object.

Suppose $\phi: G\rightarrow H$ is a group homomorphism, then we have the following diagram:

\[\xymatrix{ \rep{H}\ar[d]_-{\phi^*} \ar[r]^-{\chi_H} & \modd{Z(k H)}\ar@{-->}[d]^-{\chi_G\circ \phi^*\circ \chi_H^{-1}} \\  \rep{G} \ar[r]_-{\chi_G} & \modd{Z(k G)}.}\]

Here $\phi^*$ is a tensor functor.  It is an interesting question to understand the functor given by the dashed arrow \[F:=\chi_G\circ \phi^*\circ \chi_H^{-1}.\]

Notice that the functor $F$ is in general \emph{not} a tensor functor:  For example, let $G=\{1\}$ and let $H$ be any group with $\#\hat H\geq 2$; let $\phi: G\rightarrow H$ be the inclusion of the identity element.  If $\sigma_1$ and $\sigma_2$ are non-isomorphic irreducible $H$-representations, then we have \[\chi_H(\sigma_1)\otimes \chi_H(\sigma_2)=0\in \modd{Z(kH)}\] but $\phi^*(\sigma_1)$ and $\phi^*(\sigma_2)$ are both trivial $G$-representations, and we have \[\chi_G\phi^*(\sigma_1)\otimes \chi_G\phi^*(\sigma_2)\neq 0\in \modd{Z(kG)}.\]

\subsubsection{}  Consider now the special case when we have a group \emph{automorphism} $\phi: G\rightarrow G$.  Then we have

\[\xymatrix{ \rep{G}\ar[d]_-{\phi^*} \ar[r]^-{\chi} & \modd{Z(k G)}\ar@{-->}[d]^-{\chi\circ \phi^*\circ \chi^{-1}=F} \\  \rep{G} \ar[r]_-{\chi} & \modd{Z(k G)}.}\]

Since $\phi$ sends any conjugacy class of $G$ into a conjugacy class, it defines an automorphism \[\Phi:\hat G\cong \mathrm{Spec}(Z(k G))\longrightarrow \mathrm{Spec}(Z(k G))\cong \hat G\] by pre-composing characters with $\phi$.

Let $\rho_i$ be an irreducible $G$-representation.  Then $F(\rho_i)$ is the character of the representation $\rho_i\circ \phi$, and so we have \[\Phi(\chi(\rho_i))=\chi(\rho_i\circ\phi).\]  Therefore we have \[F=\Phi^*.\]  In particular we see that $F$ is a tensor functor.  Hence we conclude:

\begin{lem}\label{lem:comparison}  Let $\phi:G\rightarrow G$ be a group automorphism of a finite group $G$.  Then:\begin{enumerate}[(i)]
\item $\phi$ induces a tensor functor \[(\rep{G},\otimes_G)\longrightarrow(\rep{G},\otimes_G)\] defined by $\rho\mapsto \rho\circ\phi$.
\item $\phi$ induces a tensor functor \[(\rep{G},\otimes_Z)\longrightarrow(\rep{G},\otimes_Z)\] defined as the functor $F$ above.
\end{enumerate}  \end{lem}

\subsection{$G$-gerbes and their duals}

\subsubsection{}\label{}  Let $k$ be a field, let $G$ be a finite group, and let $A=\rep{G}$.  Recall from Section \ref{par:GZ} that we have two abelian tensor categories $A_G=(A,\otimes_G)$ and $A_Z=(A,\otimes_Z)$ with identical underlying abelian categories.

Then we have maps \[\xymatrix{ &&\mathrm{Aut}^\otimes(A_G)\\ G\ar[r]^-\iota & \mathrm{Aut}(G)\ar[ur]^-\kappa \ar[dr]_-{\kappa'}\\  &&\mathrm{Aut}^\otimes(A_Z),}\]where the first arrow is the inner automorphism map \[\iota(\beta):x\mapsto \beta^{-1}x\beta.\]

The map $\kappa$ sends an automorphism $\alpha:G\rightarrow G$ to $(\rho\mapsto \rho\circ\alpha)$ for every representation $\rho:G\rightarrow GL(V_\rho)$ in $A$; notice that it is an \emph{anti-}homomorphism.  Finally, $\kappa'$ factors through $\kappa$, and is defined using Lemma \ref{lem:comparison} by choosing $H=G$ and $\phi=\alpha\in \mathrm{Aut}(G)$.

\begin{lem}\label{lem:mu}  Let $\beta\in G$, and let $\alpha,\alpha'\in\mathrm{Aut}(G)$.
\begin{enumerate}[(i)]
\item There is a natural isomorphism $\mu:\kappa(\alpha\circ\iota(\beta))\rightarrow \kappa(\alpha)$.  More precisely, $\mu_\rho=\rho\alpha(\beta)$.
\item $\mu\circ \kappa(\alpha'):\kappa(\alpha'\circ\alpha\circ\iota(\beta))\rightarrow \kappa(\alpha'\circ\alpha)$ is given by $(\mu\circ \kappa(\alpha'))_\rho=\rho\alpha'\alpha(\beta)$.
\item $\kappa(\alpha')\circ\mu: \kappa(\alpha\circ\iota(\beta)\circ\alpha')\rightarrow \kappa(\alpha\circ\alpha')$ is given by $(\kappa(\alpha')\circ\mu)_\rho=\rho\alpha(\beta)$.
\item There is a natural isomorphism $\mu':\kappa'(\alpha\circ\iota(\beta))\rightarrow \kappa'(\alpha)$.  More precisely, $\mu'_\rho=\chi_\rho(\alpha(\beta))$ where $\chi_\rho$ is the character of $\rho$.
\end{enumerate}
\end{lem}

\begin{proof}  (i) For every $\rho\in A$, we need an isomorphism \[\mu_\rho: \rho\circ \alpha\circ\iota(\beta)\longrightarrow \rho\circ\alpha\] of representations.  The following commutative diagram gives the result:

\[\xymatrix{V \ar[rr]^-{\rho\alpha\iota(\beta)(x)} \ar[dd]_-{\rho\alpha(\beta)}&& V \ar[dd]^-{\rho\alpha(\beta)} \\ \\ V \ar[rr]_-{\rho\alpha(x)} && V,}\]

for every $x\in G$, where $V=V_\rho=V_{\rho\circ\alpha}=V_{\rho\circ\alpha\circ\iota(\beta)}$.

(ii) and (iii) are straightforward.

(iv)  Recall that we have an equivalence between abelian tensor categories \[\chi:A_Z=(\rep{G},\otimes_Z)\longrightarrow (\modd{k^{\#\hat G}},\otimes)\]
where $c$ is the number of conjugacy classes in $G$.  The map $\chi$ sends any irreducible representation $\rho$ to the one-dimensional vector space spanned by its character $\chi_\rho$.

Applying $\chi$ to the natural transformation $\mu$ gives the result.  \end{proof}

\subsubsection{}\label{par:coc}  Suppose now we are given elements $\alpha_{ij}\in\mathrm{Aut}(G)$ and $\beta_{ijk}\in G$ satisfying the $2$-cocycle conditions:
\begin{equation}\label{eq:coc1}\alpha_{jk}\circ \alpha_{ij}=\alpha_{ik}\circ \iota(\beta_{ijk}),\end{equation} and
\begin{equation}\label{eq:coc2}\alpha_{ij}^{-1}(\beta_{jkl})\beta_{ijl}=\beta_{ijk}\beta_{ikl}. \end{equation}

Let $\lambda_{ij}=\kappa(\alpha_{ij})\in\mathrm{Aut}^\otimes(A_G)$; similarly, let $\lambda_{ij}'=\kappa'(\alpha_{ij})\in\mathrm{Aut}^\otimes(A_Z)$.

The first of the $2$-cocycle conditions gives for every $i,j,k$ a natural isomorphism \[\mu_{ijk}: \lambda_{ij}\circ \lambda_{jk}=\kappa(\alpha_{ik}\circ\iota(\beta_{ijk}))\Longrightarrow \lambda_{ik}\]
between tensor autoequivalences on $A_G$, by taking $\alpha=\alpha_{ik}$ and $\beta=\beta_{ijk}$ in Lemma \ref{lem:mu}(i).

Similarly, by Lemma \ref{lem:mu}(iv) we have natural isomorphism \[\mu'_{ijk}:\lambda'_{ij}\circ \lambda_{jk}'\Longrightarrow \lambda_{ik}'.\]

\subsubsection{}\label{}  With notations as above, we claim that the following diagram is commutative:

\[\xymatrix{ (\lambda_{ij}\circ \lambda_{jk})\circ \lambda_{kl} \ar@{=>}[rr]^-{\mu_{ijk}\circ \lambda_{kl}} \ar@{=}[d]&& \lambda_{ik}\circ \lambda_{kl} \ar@{=>}[rr]^-{\mu_{ikl}} && \lambda_{il}\ar@{=}[d]  \\  \lambda_{ij}\circ (\lambda_{jk}\circ \lambda_{kl}) \ar@{=>}[rr]_-{\lambda_{ij}\circ\mu_{jkl}} && \lambda_{ij}\circ \lambda_{jl} \ar@{=>}[rr]_-{\mu_{ijl}} && \lambda_{il}. }\]

By Theorem \ref{thm:2descent}, we have the following consequence:

\begin{cor}  Let $G$ be a finite group, let $A_G=(\rep{G},\otimes_G)$.  Let $\B\rightarrow \mathfrak S$ be a category fibred in abelian tensor categories over a site with final object $S_0$. Let $\mathcal U=\{u_i:T_i\rightarrow S_0\}$ be a covering.

Let $\alpha_{ij}\in\mathrm{Aut}(G)$ and $\beta_{ijk}\in G$ be chosen so that they satisfy the 2-cocycle conditions \eqref{eq:coc1} and \eqref{eq:coc2}.  Then \[(\uline{A_G}(\B_{/T_i}), \lambda_{ij}, \mu_{ijk})\] defined above is a 2-descent datum of stacks over $\mathfrak S$.

In particular there is a stack $\mathscr Y$ over $\mathfrak S$ satisfying $u_i^{-1}\mathscr Y\cong \uline{A_G}(\B_{/T_i})$ by Theorem \ref{thm:2descent}.  \end{cor}

The proof of the commutativity of the diagram above is a direct calculation using Lemma \ref{lem:mu}(i)-(iii) and the first cocycle condition \eqref{eq:coc1}:

\begin{align*}
(\mu_{ikl}\circ (\mu_{ijk}\circ \lambda_{kl}))_\rho &= \mu_{ikl,\rho}\circ \mu_{ijk,\lambda_{kl}(\rho)} \\
&= \rho\alpha_{il}(\beta_{ikl})\circ \rho\alpha_{kl}\alpha_{ik}(\beta_{ijk})\\
&= \rho\alpha_{il}(\beta_{ikl})\circ \rho\alpha_{il}\iota(\beta_{ikl})(\beta_{ijk})\\
&= \rho\alpha_{il}(\beta_{ijk}\beta_{ikl}),
\end{align*}

\begin{align*}
(\mu_{ijl}\circ (\lambda_{ij}\circ \mu_{jkl}))_\rho &=\mu_{ijl,\rho}\circ \lambda_{ij}(\mu_{jkl,\rho})\\
&=\rho\alpha_{il}(\beta_{ijl})\circ\rho\alpha_{jl}\alpha_{ij}\alpha_{ij}^{-1}(\beta_{jkl})\\
&=\rho\alpha_{il}(\beta_{ijl})\circ\rho\alpha_{il}\iota(\beta_{ijl})\alpha_{ij}^{-1}(\beta_{jkl})\\
&=\rho\alpha_{il}(\alpha_{ij}^{-1}(\beta_{jkl})\beta_{ijl}).
\end{align*}

Therefore the commutativity follows from the second cocycle condition \eqref{eq:coc2}.

By applying $\chi$ to the calculation above, we see that the analogous diagram with $\lambda$ replaced with $\lambda'$  and $\mu$ replaced with $\mu'$ is also commutative.  Hence we have

\begin{cor}  Let $G$ be a finite group, let $A_Z=(\rep{G},\otimes_Z)$.  Let $\B\rightarrow \mathfrak S$ be a category fibred in abelian tensor categories over a site with final object $S_0$. Let $\mathcal U=\{u_i:T_i\rightarrow S_0\}$ be a covering.

Let $\alpha_{ij}\in\mathrm{Aut}(G)$ and $\beta_{ijk}\in G$ be chosen so that they satisfy the 2-cocycle conditions \eqref{eq:coc1} and \eqref{eq:coc2}.  Then \[(\uline{A_Z}(\B_{/T_i}), \lambda'_{ij}, \mu'_{ijk})\] defined above is a 2-descent datum of stacks over $\mathfrak S$.

In particular there is a a stack $\hat{\mathscr Y}$ over $\mathfrak S$ satisfying $u_i^{-1}\hat{\mathscr Y}\cong \uline{A_Z}(\B_{/T_i})$ by Theorem \ref{thm:2descent}.  \end{cor}

\subsubsection{}\label{}  Let $\mathfrak S$ be the \'etale site of affine schemes over a given scheme $S_0$ and let $\B\rightarrow \mathfrak S$ be the fibred category of quasi-coherent sheaves.

The stack $\mathscr Y$ constructed above is a $G$-gerbe over $\mathfrak S$.  Indeed, by the proof of Theorem \ref{thm:BG} we know that each stack $\uline{A_G}(\B_{/_{T_i}})$ is isomorphic to $BG\times T_i$.  The stack $\hat{\mathscr Y}$ should be considered as its \emph{dual space} as in \cite{TT}.

\subsubsection{}\label{}  \emph{Remarks.}  The $G$-gerbes arising from the 2-descent Theorem \ref{thm:2descent} considered here are a special kind.  To cover more general $G$-gerbes one needs to consider 2-cocycles with upper indices \[(\alpha_{ij}^{r},\beta_{ijk}^{rst})\] as in \cite[2.4 and 2.7]{breen_class2gerbe}, where $T^r_{ij}\rightarrow T_{ij}$ is a covering of $T_{ij}$ with index $r$, $\alpha_{ij}^r\in\mathrm{Aut}(G)$, and $\beta_{ijk}^{rst}\in G$.

The 2-cocycle conditions are

\begin{equation}\label{eq:coc1+}\alpha_{jk}^s\circ \alpha_{ij}^r=\alpha_{ik}^t\circ \iota(\beta_{ijk}^{rst}),\end{equation} and
\begin{equation}\label{eq:coc2+}(\alpha_{ij}^r){}^{-1}(\beta_{jkl}^{swv})\beta_{ijl}^{rvu}=\beta_{ijk}^{rst}\beta_{ikl}^{twu}. \end{equation}

We illustrate the construction in this more general setting by considering the $G$-gerbe \[BH\rightarrow BQ\] arising from a short exact sequence \[1\longrightarrow G\longrightarrow H\longrightarrow Q\longrightarrow 1\] of finite groups.

In this case consider the \'etale covering $\mathrm{pt}\rightarrow BQ$ by a single affine scheme $T_1:=\mathrm{pt}=\mathrm{Spec}(k)$.  Then we have, with notations as in \ref{par:coc},  \[T_{11}=\mathrm{pt}\times_{BQ}\mathrm{pt}\cong Q\] as a set of points; set $T_{ij}^r$ to be the point corresponding to an element $r$ in $Q$.  Therefore there is only one lower index, namely $1$, and the upper indices correspond to elements in $Q$.

For every $r\in Q$ choose a lifting $\tilde r\in H$.  This gives a set map \[Q\rightarrow \mathrm{Aut}(G)\] by sending $r$ to the conjugation automorphism of $G$ by $\tilde r$.  Denote this automorphism by $\alpha_{11}^r$.

The condition \eqref{eq:coc1+} defines $\beta_{111}^{rst}$ whenever the group elements $r,s,t\in Q$ satisfy the equality $t=rs$ and gives a $2$-cocycle; that is, condition \eqref{eq:coc2+} is satisfied.  %(\footnote{This is a bit confusing, as in we don't get a full 2-cocycle.  Compare with (3.2) and (3.3) in "Duality theorems..."})

These choices of $\alpha_{11}^r$ and $\beta_{111}^{rst}$ allow us to glue $T_{11}\times BG$, that is, $\#Q$ copies of $BG$ together to get a $G$-gerbe $\mathscr Y$ over $BQ$.

Recall that each local copy of $BG$ is realized as $\uline{A_G}(\B_{/\mathrm{pt}})$, where $\mathrm{pt}$ is a point in $T_{11}$ and $\B\rightarrow \mathfrak S$ is the fibred category of quasi-coherent sheaves over the site of affine schemes over $\Lambda$.  Hence the dual $\hat{\mathscr Y}$ is glued from $\#Q$ copies of the scheme \[\hat G=\mbox{the set of isomorphism classes of irreducible representations of $G$}=\uline{A_Z}(\B_{/\mathrm{pt}})\] via the isomorphisms induced by $\alpha_{11}^r$, $r\in Q$.  In other words, $\hat {\mathscr Y}$ is the quotient of $\hat G$ by this $Q$-action %(which depends on the choice of lifting $r\mapsto \tilde r$):
\[\hat{\mathscr Y}=[\hat G/Q].\]

\subsubsection{}  Now we construct a twisted sheaf on $\hat{\mathscr Y}$ using \ref{par:gluing sheaves}.  See \cite{TT} for the role twisted sheaves played in gerbe duality.

First we consider $\hat{\mathscr Y}$, which is glued from $\uline{A_Z}(\B/_{T_i})=\hat G\times T_i$.  Consider the regular representation $\tilde \rho=\sum V_{\rho_s^*}\otimes \rho_s$, where $\{\rho_s\}$ is a set of representatives of isomorphism classes of irreducible $G$-representations.  Here $V_{\rho_s^*}\otimes \rho_s$ means the direct sum of $\dim V_{\rho_s^*}$ copies of $\rho_s$.  Denote by $\chi_s$ the character of $\rho_s$, then $\hat G\cong \{\chi_s\}$.

Let $\mathcal G_i$ be the sheaf on $\hat G\times T_i$ corresponding to $\tilde\rho\in A$ via the construction in \ref{par:sheaf}.  Then $\mathcal G_i$ is simply the sheaf on $\hat G\times T_i$ whose restriction to $\{\chi_s\}\times T_i$ is the trivial vector bundle $V_{\rho_s^*}\otimes \mathcal O_{T_i}$.

We will identify $\hat G\times T_i|_{ij}$ with $\hat G\times T_{ij}$, then we have an automorphism $\chi_{ij}$ on $\hat G\times T_{ij}$ induced by the automorphism $\alpha_{ij}\in\mathrm{Aut}(G)$.  By the cocycle condition (\ref{eq:coc1}) and the fact that an inner automorphism acts trivially on the set $\hat G$ of characters, we see that the natural isomorphism \[\phi_{ijk}: \chi_{jk}\circ \chi_{ij}\Rightarrow \chi_{ik}\] between stack isomorphisms is the identity:  In fact every $\chi_{ij}$ is the scheme automorphism on $\hat G\times T_{ij}$ given by the action of $\alpha_{ij}$ acting on $\hat G$.

The sheaf $\chi_{ij}^*(\mathcal G_j|_{ij})$ is the the sheaf whose restriction to $\{\chi_s\}\times T_{ij}$ is the trivial vector bundle $V_{\rho^*_v}\otimes \mathcal O_{T_{ij}}$, where $\rho_v$ is the representative of the isomorphism class of $\rho_s\circ \alpha_{ij}$.

The isomorphism $\rho^*_s\circ \alpha_{ij}\cong \rho^*_v$ is given by a vector space isomorphism (unique only up to a non-zero scalar) \[\tau_{ij,s}: V_{\rho_s^*}\stackrel{\cong}{\longrightarrow} V_{\rho_v^*}.\]  This defines sheaf isomorphisms \[\delta_{ij}: \mathcal G_i|_{ij}\stackrel{\cong}{\longrightarrow} \chi_{ij}^*(\mathcal G_j|_{ij})\] between sheaves on $\hat G\times T_{ij}$, and hence we have constructed a twisted sheaf $\mathcal G$ on $\hat{\mathscr Y}$ associated to the regular representation $\tilde \rho$.

To compute the twisting, which is a 2-cycle with values in $\mathcal O_{S_0}^\times$, first recall that by the cocycle condition (\ref{eq:coc1}) we have \[\alpha_{ki}\circ \alpha_{jk}\circ \alpha_{ij}=\iota(\beta_{ijk})\] in $\mathrm{Aut}(G)$, where $\beta_{ijk}\in G$, and for any $\beta\in G$, $\iota(\beta)$ denotes the inner automorphism $x\mapsto \beta^{-1}x\beta$ on $G$.  In particular, the isomorphism of vector spaces $\rho_s^*(\beta^{-1})$ gives an isomorphism from $\rho_s$ to $\rho_s\circ\iota(\beta)$.  Now consider the following commutative diagram:

\[\xymatrix{ \rho_s^* \ar[d]_-{\rho_s^*(\beta^{-1}_{ijk})} \ar@{-->}@/^15pt/[ddrrr]^-{c_{s,ijk}}\\ \rho_s^*\circ \iota(\beta_{ijk}) \ar@{=}[d] \\ \rho_s^*\circ \alpha_{ki}\circ \alpha_{jk}\circ \alpha_{ij} \ar[r]_-{\tau_{ki,s}} & \rho^*_t\circ \alpha_{jk}\circ \alpha_{ij} \ar[r]_-{\tau_{jk,t}} & \rho_u^*\circ\alpha_{ij} \ar[r]_-{\tau_{ij,u}} & \rho_s^*,  } \]  where $c_{s,ijk}$ is a non-zero scalar:  Indeed, we can show that the composition of these vector space isomorphisms gives an automorphism of the \emph{representation} $\rho^*_s$.  The sheaf $\mathcal G$ is $\{c_{s,ijk}\}$-twisted.

%(\footnote{Need to check the computation of the part "we can show...".})

\subsubsection{}  An analogous construction as above gives a (non-twisted) sheaf on $\mathscr Y$, which is glued from $\uline{A_G}(\B_{/T_i})=BG\times_{S_0}T_i$.  Again consider the regular representation $\tilde\rho$, and let $\mathcal F_i$ be the sheaf on $BG\times_{S_0}T_i$ corresponding to $\tilde\rho$.

If $\mathcal M\rightarrow S$ is a $G$-torsor over a scheme $S$ over $T_i$, namely an object in $BG\times_{S_0}T_i$, then \[\mathcal F_i: (\mathcal M\rightarrow S)\mapsto \Gamma(S,\tilde \rho_*\mathcal M)\] where $\tilde \rho_*\mathcal M$ is the flat vector bundle on $S$ given by the 1-cocycle with values in $GL(V_{\tilde\rho})$ by pushing-forward the 1-cocycle with values in $G$ corresponding to $\mathcal M$.

We will identify $BG\times_{S_0}T_i|_{ij}$ with $BG\times_{S_0}T_{ij}$, and then we have an automorphism $\chi_{ij}$ on it induced by $\alpha_{ij}\in\mathrm{Aut}(G)$.  Since inner automorphisms on $G$ act trivially on $BG$, the natural isomorphism \[\phi_{ijk}:\chi_{jk}\circ \chi_{ij}\Rightarrow \chi_{ik}\] between stack isomorphisms is again the identity.

The sheaf $\chi_{ij}^*(\mathcal F_j|_{ij})$ is given by \[(\mathcal M\rightarrow S)\mapsto \Gamma(S,(\tilde\rho\circ\alpha_{ij})_*\mathcal M).\]

Notice that $V_{\tilde\rho}$ is isomorphic to $kG$ as a $k$-vector space, and so any group automorphism $\alpha\in\mathrm{Aut}(G)$ induces a \emph{linear} automorphism of $V_{\tilde\rho}$.  In particular, we then have a natural isomorphism \[\delta_{ij}: \mathcal F_i|_{ij}\stackrel{\cong}{\longrightarrow} \chi_{ij}^*(\mathcal F_j|_{ij}).\]

Finally, the vector bundles $\tilde\rho_*\mathcal M$ and $(\tilde\rho\circ \iota(\beta_{ijk}))_*\mathcal{M}$ are identified, as in the case of $\hat{\mathscr Y}$, by $\beta^{-1}_{ijk}$, which is exactly the inverse of the composition of $\delta_{ij}$'s over $T_{ijk}$.  Therefore the twisting is the identity in this case.  In other words the twisted sheaf $\mathcal F$ we obtained is actually a sheaf.

%(\footnote{Need to check a few details.})

\appendix

\section{Twisting by a 2-cocycle}\label{sec:twisting}

\subsection{2-descent data}

\subsubsection{}  In this appendix we fix some notations and recall the $2$-descent of stacks.  Some of the materials below can be found in \cite{breen_class2gerbe}.

\subsubsection{}\label{}  Let $\mathfrak S$ be a site with final object $S_0$, then we have an equivalence $\mathfrak S\cong\mathfrak S_{/S_0}$.  Fix any covering $\mathcal U=\{u_i:T_i\rightarrow S_0\}$ of $S_0$, then any stack $\mathscr X$ over $\mathfrak S$ gives rise to a stack \[\mathscr X|_i:=u_i^{-1}\mathscr X\cong \mathscr X\times_{\mathfrak S}\mathfrak S_{/T_i}\] over $\mathfrak S_{/T_i}$ for every $i$, along with isomorphisms \[can_{ij}:p_1^{-1}u_i^{-1}\mathscr X\stackrel{\cong}{\longrightarrow} p_2^{-1}u_j^{-1}\mathscr X\] of stacks over $\mathfrak S_{/T_{ij}}$, where $T_{ij}=T_i\times_{S_0}T_j$, and $p_1$ and $p_2$ denote the two projections to $T_i$ and $T_j$ respectively.  To simplify notations, we will write $\mathscr X|_i|_{ij}$ for $p_1^{-1}u_i^{-1}\mathscr X$, etc; in particular we have \[can_{ij}:\mathscr X|_i|_{ij}\stackrel{\cong}{\longrightarrow} \mathscr X|_j|_{ij}.\]  We will similarly use the restriction notation for the pull-back functors.

The fibred category structure on $\mathscr X$, or more precisely the natural isomorphisms relating different pull-backs gives moreover natural isomorphisms \[\phi_{ijk}: (can_{jk}|_{ijk})\circ (can_{ij}|_{ijk})\Longrightarrow can_{ik}|_{ijk}\] between isomorphisms of stacks over $\mathfrak S_{/T_{ijk}}$, where $T_{ijk}=T_i\times_{S_0}T_j\times_{S_0}T_k$.

These natural transformations satisfy a compatibility cocycle condition (see below).

\subsubsection{}  The situation above formalizes to the notion of {\em $2$-descent data}:  A $2$-descent datum of stacks over $\mathfrak S$ with respect to the covering $\mathcal U$ is a triple $(\mathscr X_i,\chi_{ij}, \phi_{ijk})$ of stacks $\mathscr X_i$ over $\mathfrak S_{/T_i}$, stack isomorphisms \[\chi_{ij}:\mathscr X_i|_{ij}\stackrel{\cong}{\longrightarrow} \mathscr X_j|_{ij},\]and natural isomorphisms \[\phi_{ijk}: (\chi_{jk}|_{ijk})\circ (\chi_{ij}|_{ijk})\Longrightarrow \chi_{ik}|_{ijk}\] between functors from $\mathscr X_i|_{ijk}$ to $\mathscr X_k|_{ijk}$.  These are required to  satisfy the condition that for every $i,j,k,l$ the following diagram of functors from $\mathscr X_i|_{ijkl}$ to $\mathscr X_l|_{ijkl}$ is commutative:

\[\xymatrix{ \chi_{kl}\circ (\chi_{jk} \circ \chi_{ij})\ar@{=>}[rr]^-{\chi_{kl}( \phi_{ijk})} \ar@{=}[d]  && \chi_{kl}\circ \chi_{ik} \ar@{=>}[rr]^-{\phi_{ikl}} && \chi_{il} \ar@{=}[d]\\
(\chi_{kl}\circ \chi_{jk}) \circ \chi_{ij}\ar@{=>}[rr]_-{\phi_{jkl}\circ \chi_{ik}}  && \chi_{jl}\circ \chi_{ij} \ar@{=>}[rr]_-{\phi_{ijl}} && \chi_{il},
}\]  where for legibility we have omitted $|_{ijkl}$ throughout.

The following is part of \cite[Example~1.11~(i)]{breen_class2gerbe}.%; see also \cite[Remark on page~14]{breen_class2gerbe}.
\begin{thm}\label{thm:2descent}  Given a site $\mathfrak S$ with final object $S_0$ and a covering $\mathcal U=\{u_i: T_i\rightarrow S_0\}$, suppose we have a $2$-descent datum $(\mathscr X_i,\chi_{ij},\phi_{ijk})$ of stacks with respect to $\mathcal U$ as above.  Then there is a stack $\mathscr X$ over $\mathfrak S$ along with isomorphisms \[\mathscr X|_i=u_i^{-1}\mathscr X\cong \mathscr X_i\] of stacks over $\mathfrak S_{/T_i}$.  \end{thm}

\subsubsection{}\label{par:gluing sheaves}  Now we explain how to construct twisted sheaves on 2-descended stacks.  Let $(\mathscr X_i,\chi_{ij},\phi_{ijk})$ be a 2-descent datum of stacks over $\mathfrak S$ with respect to a covering $\mathcal U=\{u_i:T_i\rightarrow S_0\}$ of the final object $S_0\in\mathfrak S$.

Let $\mathcal F_i$ be a sheaf on $\mathscr X_i$.  Suppose we are further given sheaf isomorphisms \[\delta_{ij}: \mathcal F_i|_{ij}\stackrel{\cong}{\longrightarrow} \chi_{ij}^*(\mathcal F_j|_{ij}),\]where $\chi_{ij}^*(\mathcal F_j|_{ij})=\mathcal F_i|_{ij}\circ \chi_{ij}$ as a (set-valued) functor.  Then, omitting $|_{ijk}$ throughout, we have isomorphisms of sheaves on $\mathscr X_i|_{ijk}$:

\[\xymatrix{\mathcal F_i \ar@/_3ex/[drrrrrr]_-{\eta^i_{ijk}} \ar[rr]^-{\delta_{ij}} && \chi_{ij}^*(\mathcal F_j)  \ar[rr]^-{\chi_{ij}^*(\delta_{jk})} && \chi^*_{ij}\chi^*_{jk}(\mathcal F_k) \ar[rr]^-{\chi_{ij}^*\chi_{jk}^*(\delta_{ki})} && \chi_{ij}^*\chi_{jk}^*\chi_{ki}^*(\mathcal F_i) \\ &&&&&& \mathcal F_i\ar[u]_-\cong, }\]  where the vertical isomorphism is induced from the natural isomorphism \[\phi_{ijk}:\chi_{jk}\circ \chi_{ij}\Rightarrow \chi_{ik}.\] (And the normalization $\chi_{ik}=\chi_{ki}^{-1}$.)

Here $\eta^i_{ijk}$ is an automorphism of the sheaf $\mathcal F_i|_{ijk}$ on $\mathscr X_i|_{ijk}$.  There is a natural compatibility relation between $\eta^i_{ijk}$ and $\eta^j_{ijk}$ in terms of the isomorphism $\delta_{ij}$.  If for every $i,j,k$ the automorphism $\eta^i_{ijk}$ is the identify automorphism, then the datum \[(\mathcal F_i,\delta_{ij})\] defines a sheaf on the stack $\mathscr X$.  In general, $\eta^i_{ijk}$ needs not be the identity, and we get a twisted sheaf on $\mathscr X$.

\subsection{2-cocycle on an abelian category}

\subsubsection{}\label{}  Now let $A$ be an abelian tensor category, and let $\pi:\B\rightarrow \mathfrak S$ be a category fibred in abelian tensor categories satisfying the conditions of Theorem \ref{thm:stack}.  Notice that for every $i$ the restriction \[\pi_i:\B|_i:=\B_{/T_i}\rightarrow \mathfrak S_{/T_i}\] also satisfies the conditions of Theorem \ref{thm:stack}, and therefore we have a stack $\uline{A}(\B|_i)$ over $\mathfrak S_{/T_i}$ for every $i$ which is isomorphic to $\uline{A}(\B)|_i=u_i^{-1}\uline{A}(\B)$.

\subsubsection{}\label{par:glueing}  With notations as above, we would like to glue the stacks $\uline A(\B|_i)$ using $2$-descent to produce new stacks over $\mathfrak S$.

To this end, let $\mathrm{Aut}^\otimes(A)$ denote the set of autoequivalences of $A$ which are tensor functors.  Choose $\lambda_{ij}\in\mathrm{Aut}^\otimes(A)$ for every $i,j$ and natural isomorphisms \[\mu_{ijk}: \lambda_{ij}\circ \lambda_{jk}\Longrightarrow \lambda_{ik},\]satisfying the ``tetrahedron'' condition that the following diagram of functors is commutative for every $i,j,k,l$:

\[\xymatrix{ (\lambda_{ij}\circ \lambda_{jk})\circ \lambda_{kl} \ar@{=>}[rr]^-{\mu_{ijk}\circ \lambda_{kl}} \ar@{=}[d]&& \lambda_{ik}\circ \lambda_{kl} \ar@{=>}[rr]^-{\mu_{ikl}} && \lambda_{il}\ar@{=}[d]  \\  \lambda_{ij}\circ (\lambda_{jk}\circ \lambda_{kl}) \ar@{=>}[rr]_-{\lambda_{ij}\circ\mu_{jkl}} && \lambda_{ij}\circ \lambda_{jl} \ar@{=>}[rr]_-{\mu_{ijl}} && \lambda_{il}. }\]

These conditions imply that the triple \[(\uline{A}(\B|_i), \lambda_{ij}^*,\mu_{ijk})\] is a $2$-descent datum of stacks over $\mathfrak S$ with respect to the covering $\mathcal U$.  Therefore by Theorem \ref{thm:2descent} there is a stack $\mathscr X$ over $\mathfrak S$ satisfying \[\mathscr X|_i=u_i^{-1}\mathscr X\cong \uline{A}(\B).\]

%the following is not necessary

Here $\lambda^*_{ij}$ denotes the isomorphism \[\lambda^*_{ij}: \uline{A}(\B|_{i})|_{ij}\longrightarrow \uline{A}(\B|_{j})|_{ij}\] between stacks over $\mathfrak S_{/T_{ij}}$ defined by \[(f,(S,F)) \mapsto (f,(S,F\circ \lambda_{ij}))\] for every $f:S\rightarrow T_{ij}$; here $F:A\rightarrow \B|_{i,S}\cong\B_{S}$ is a tensor functor.

\subsubsection{}\label{}  \emph{Remark.}  It may be interesting to study the set of such $2$-cocycles for given $A$.  It should form some sort of cohomology set in degree $2$.

%bibliography

\end{document}